\documentclass[12pt]{amsart}
\usepackage{a4wide}
\usepackage{amssymb}
\usepackage{amsmath}
\usepackage{amsfonts}
\usepackage{amsthm}
\usepackage[normalem]{ulem}
\usepackage{graphicx}
\usepackage{epsfig}
\graphicspath{{./Figures/}}
\usepackage{longtable}
\usepackage{paralist}
\usepackage[usenames,dvipsnames]{color}
\usepackage{hyperref}
\usepackage{caption}
\usepackage[labelformat=simple,labelfont={}]{subfig}

\usepackage{alphalph}
\makeatletter
\newalphalph{\fnsymbolmult}[mult]{\@fnsymbol}{5}
\makeatother

\newtheorem{thm}{Theorem}%[section]
\newtheorem{lem}[thm]{Lemma}
\newtheorem{cor}[thm]{Corollary}
\newtheorem{prop}[thm]{Proposition}

\theoremstyle{definition}
\newtheorem{defn}[thm]{Definition}

\theoremstyle{remark}
\newtheorem*{rmk}{Remark}

\newcommand{\eps}{\varepsilon}

\newcommand{\DEF}{{:=}}

\newcommand{\Sp}{\mathbb{S}}

\newcommand{\IL}{\mathbb{L}}

\newcommand{\CF}{\mathrm{CF}}

\newcommand{\PT}[1]{\mathbf{#1}}
\DeclareMathOperator{\dd}{\mathrm{d}}
\newcommand{\xctint}{I}
\newcommand{\numint}{Q}

\DeclareMathOperator{\gammafcn}{\Gamma}

\DeclareMathOperator{\diam}{diam}

\DeclareMathOperator{\wce}{wce}

\DeclareMathOperator{\HyperF}{F}

\newcommand{\Hypergeom}[5]{{\sideset{_#1}{_#2}\HyperF\!\left(\substack{\displaystyle#3\\\displaystyle#4};#5\right)}}

\newcommand{\Pochhsymb}[2]{{\left(#1\right)_{#2}}}

\allowdisplaybreaks[1]

\title[QMC designs: optimal order Quasi Monte Carlo Integration on the sphere]{QMC designs: optimal order Quasi Monte Carlo Integration schemes on the sphere}
\author[J. S. Brauchart, E. B. Saff, I. H. Sloan, and R. S. Womersley]{J. S. Brauchart, E. B. Saff, I. H. Sloan, and R. S. Womersley} %\textasteriskcentered
\thanks{\noindent This research was supported by an Australian Research Council Discovery Project. \\
The research of the second author was also supported by U.S. National Science Foundation grants DMS-0808093 and DMS-1109266. }

\date{\today}

\DeclareCaptionLabelFormat{andtable}{#1˜\ #2 and \tablename˜\ \thetable}

\begin{document}

\address{J. S. Brauchart, I. H. Sloan, and R. S. Womersley:
School of Mathematics and Statistics,
University of New South Wales,
Sydney, NSW, 2052,
Australia }
\email{j.brauchart@unsw.edu.au}
\email{r.womersley@unsw.edu.au}
\email{i.sloan@unsw.edu.au}
\address{E. B. Saff:
Center for Constructive Approximation,
Department of Mathematics,
Vanderbilt University,
Nashville, TN 37240,
USA}
\email{edward.b.saff@vanderbilt.edu}

\begin{abstract}
We study equal weight numerical integration, or Quasi Monte Carlo (QMC) rules, for functions in a Sobolev space $\mathbb{H}^s( \mathbb{S}^d)$ with smoothness parameter $s > d/2$ defined over the unit sphere $\mathbb{S}^d$ in $\mathbb{R}^{d+1}$. Focusing on $N$-point configurations that achieve optimal order QMC error bounds (as is the case for efficient  spherical designs), we are led to introduce the concept of QMC designs: these are sequences of $N$-point configurations $X_N$ on $\mathbb{S}^d$ such that the worst-case error satisfies
\begin{equation*} %\label{eq:asymptotic.bound}
\sup_{\substack{f \in \mathbb{H}^s( \mathbb{S}^d ), \\ \| f \|_{\mathbb{H}^s} \leq 1}} \Bigg| \frac{1}{N} \sum_{\mathbf{x} \in X_N} f( \mathbf{x} ) - \int_{\mathbb{S}^d} f( \mathbf{x} ) \, \mathrm{d} \sigma_d( \mathbf{x} ) \Bigg| = \mathcal{O}\big( N^{-s/d} \big), \qquad N \to \infty, 
\end{equation*}
with an implied constant that depends on the $\mathbb{H}^s( \mathbb{S}^d )$-norm, but is independent of $N$.  
Here $\sigma_d$ is the normalized surface measure on $\mathbb{S}^d$.

{\sloppy
We provide methods for generation and numerical testing of QMC designs.
An essential tool is an expression for the worst-case error 
in terms of a reproducing kernel for the space $\mathbb{H}^s( \mathbb{S}^d )$ with $s > d/2$. As a consequence of this and a recent result of Bondarenko et al. on the existence of spherical designs with appropriate number of points, we show that minimizers of the $N$-point energy for this kernel form a sequence of QMC designs for $\mathbb{H}^s( \mathbb{S}^d )$. Furthermore, without appealing to the Bondarenko et al. result,  we prove that point sets that maximize the sum of suitable powers of the Euclidean distance between pairs of points form a sequence of QMC designs for $\mathbb{H}^s( \mathbb{S}^d )$ with $s$ in the interval ${(d/2,d/2+1)}$. For such spaces there exist reproducing kernels with simple closed forms that are useful for numerical testing of optimal order Quasi Monte Carlo integration.}

Numerical experiments suggest that many familiar sequences of point sets on the sphere (equal area points, spiral points, minimal [Coulomb or logarithmic] energy points, and Fekete points) are QMC designs for appropriate values of $s$.
For comparison purposes we show that configurations of random points that are independently and uniformly distributed on the sphere do not constitute QMC designs for any $s>d/2$.  

If $(X_N)$ is a sequence of QMC designs for $\mathbb{H}^s( \mathbb{S}^d)$, we prove that it is also a sequence of QMC designs for  $\mathbb{H}^{s'}( \mathbb{S}^d)$  for all $s'\in (d/2,s)$. This leads to the question of determining the supremum of such $s$, for which we provide estimates based on computations for the aforementioned sequences.
\end{abstract}


\maketitle

% \MARKED{Red}{In red hints for authors. Things which should be considered very carefully.}

\section{Introduction}
\label{sec:intro}

% [Ian's changes in \MARKED{Cyan}{cyan}. Ed and Johann's changes in other colors]

In this paper we introduce a new notion for sequences of finite point sets on the unit sphere $\mathbb{S}^d$ in the Euclidean space $\mathbb{R}^{d+1}$, $d \geq 2$, namely that of {\em sequences of QMC designs}. These are sequences that emulate spherical designs in that they provide optimal order equal weight numerical integration (or Quasi Monte Carlo) rules for certain Sobolev spaces of functions over the unit sphere $\mathbb{S}^d$.

A spherical $t$-design, a concept introduced in the groundbreaking paper \cite{DeGoSei1977} by Delsarte, Goethals and Seidel, is a finite subset $X_N \subset\mathbb{S}{^d}$ with the characterizing property that an equal weight integration rule with nodes from $X_N$ integrates exactly all polynomials $P$ with degree $\leq t$; that is,
\begin{equation} \label{eq:spherical.designs}
\frac{1}{N} \sum_{\PT{x} \in X_N} P( \PT{x})= \int_{\mathbb{S}^d} P( \PT{x} ) \, \dd \sigma_d( \PT{x} ), \qquad \deg P \leq t.
\end{equation}
Here $N=|X_N|$ is the cardinality of $X_N$, or the number of points of the spherical design, while the integral is with respect to the normalized 
% (Hausdorff) 
surface measure $\sigma_d$ on $\mathbb{S}^d$, and the polynomials of degree $\le t$ are the restrictions to $\mathbb{S}^d$ of the polynomials of degree $\le t$ on $\mathbb{R}^{d+1}$.

Sequences of spherical designs have a known fast-convergence property in Sobolev spaces.  (See Section \ref{subsec:sobolev} below for the definition of the Sobolev space $\mathbb{H}^s( \mathbb{S}^d$).) This property, stated in the following theorem, was first proved for the particular case $s=3/2$ and $d=2$ in \cite{HeSl2005}, then extended to all $s>1$ for $d=2$ in \cite{HeSl2006}, and finally extended to all $s>d/2$ and all $d\ge 2$ in \cite{BrHe2007}. (The results in those papers were proved for all positive-weight integration rules with an appropriate degree of polynomial accuracy in relation to the number of points, but here we restrict our attention to equal weight rules.)

\begin{thm}\label{thm:s.d.prop}
% Given $d\ge 2$ and $s>d/2$
Given $s>d/2$, there exists $C(s,d)>0$ depending on the $\mathbb{H}^s( \mathbb{S}^d )$-norm such that for every $N$-point spherical $t$-design $X_N$ on $\mathbb{S}^d$ there holds
\begin{equation} \label{eq:asymptotic_in_t.bound}
\sup_{\substack{f \in \mathbb{H}^s( \mathbb{S}^d ), \\ \| f \|_{\mathbb{H}^s} \leq 1}} \Bigg| \frac{1}{N} \sum_{\PT{x} \in X_N} f( \PT{x} ) - \int_{\mathbb{S}^d} f( \PT{x} ) \dd \sigma_d( \PT{x} ) \Bigg| \leq \frac{C(s,d)}{t^s}.
\end{equation}
\end{thm}
Note that the constant $C(s,d)$ in this theorem does not depend on $t$ or on $N$, or on the particular spherical design $X_N$.  Note too that the condition $s>d/2$ is a natural one, since by the Sobolev embedding theorem this is the condition needed for $\mathbb{H}^s( \mathbb{S}^d)$ to be embedded in $C(\mathbb{S}^d)$.

The relation between $N$ and $t$ in a spherical design is not fixed, but there are known lower bounds on $N$ (see \eqref{eq:DGS-bound} below) that tell us that $N$ is at least of order $t^d$, and a recent result \cite{BoRaVi2011arXiv} of Bondarenko et al. 
% (see Section \ref{sec:sd}) 
(see Theorem~\ref{thm:bondarenko.et.al} below) 
asserts that given $t$ there always exists a spherical design with $N \asymp t^d$. Here we write $a_n\asymp b_n$ to mean that there exist positive constants $c_1$ and $ c_2$ independent of $n$ such that $c_1 a_n\le b_n\le c_2 a_n$ for all $n$.

Motivated by these facts and the belief that the only interesting sequences of spherical designs are those with $N\asymp t^d$, we now define the notion of a sequence of QMC designs.
\begin{defn} \label{def:specific_approx.sph.design}
Given $s>d/2$, a sequence $(X_N)$ of $N$-point configurations  on $\mathbb{S}^d$ with $N\to\infty$ is said to be a {\em sequence of QMC designs for $\mathbb{H}^s( \mathbb{S}^d )$} if there exists $c(s,d)>0$, independent of $N$, such that
\begin{equation}\label{eq:approxfors}
\sup_{\substack{f \in \mathbb{H}^s( \mathbb{S}^d ), \\ \| f \|_{\mathbb{H}^s} \leq 1}} \Bigg| \frac{1}{N} \sum_{\PT{x} \in X_N} f( \PT{x} ) - \int_{\mathbb{S}^d} f( \PT{x} ) \dd \sigma_d( \PT{x} ) \Bigg|  \leq \frac{c(s,d)}{N^{s/d}}.
\end{equation}
\end{defn}
In this definition $X_N$ need not be defined for all natural numbers $N$: it is sufficient that $X_N$ exists for an infinite subset of the natural numbers.
By a special case of theorems in \cite{HeSl2005b} and \cite{He2006}, the exponent of $N$ in \eqref{eq:approxfors} cannot be larger than $s/d$:
% It is important to note that the exponent of $N$ in \eqref{eq:approxfors} cannot be larger than $s/d$, as a special case of theorems proved in \cite{HeSl2005b} and \cite{He2006} asserts:  

\begin{thm} \label{thm:wce.lower.bound}
Given $s>d/2$, there exists $c^\prime(s,d)>0$ depending on the $\mathbb{H}^s( \mathbb{S}^d )$-norm such that for any $N$-point configuration on $\mathbb{S}^d$ 
\begin{equation} \label{eq:wce.lower.bound}
\frac{c^\prime(s,d)}{N^{s/d}} \leq \sup_{\substack{f \in \mathbb{H}^s( \mathbb{S}^d ), \\ \| f \|_{\mathbb{H}^s} \leq 1}} \Bigg| \frac{1}{N} \sum_{\PT{x} \in X_N} f( \PT{x} ) - \int_{\mathbb{S}^d} f( \PT{x} ) \dd \sigma_d( \PT{x} ) \Bigg|.
\end{equation}
\end{thm}

The following theorem, obtained by appealing to results of Brandolini et al.~\cite{BrChCoarXiv1012.5409v1}, asserts that if $(X_N)$ is a sequence of QMC designs for $\mathbb{H}^{s}(\mathbb{S}^d)$, then it is also so for all coarser Sobolev spaces $\mathbb{H}^{s^\prime}( \mathbb{S}^d )$ with $d/2 < s^\prime < s$. Like all statements in the paper needing formal proof, it is  established in Section \ref{sec:proofs}.

\begin{thm}\label{thm:rough}
Given $s>d/2$, let $(X_N)$ be a sequence of QMC designs for $\mathbb{H}^s(\mathbb{S}^d)$. Then $(X_N)$ is a sequence of QMC designs for $\mathbb{H}^{s'}(\mathbb{S}^d)$, for all $s'$ satisfying $d/2 < s' \le s$.
\end{thm}

It follows from this theorem that for every sequence of QMC designs $(X_N)$ there is some number $s^*$ such that $(X_N)$ is a sequence of QMC designs for all $s$ satisfying $d/2< s < s^*$, and is not a QMC design for $s>s^*$; that is
\begin{equation} \label{eq:s.star}
s^* \DEF  s^*[(X_N)] \DEF \sup\big\{ s : \text{$(X_N)$ is a sequence of QMC designs for $\mathbb{H}^s( \mathbb{S}^d )$}  \big\}.
\end{equation}
If $s^* = + \infty$, we append the adjective ``generic.''

\begin{defn} \label{def:approx.sph.design}
A sequence of $N$-point configurations $(X_N)$ on $\mathbb{S}^d$  is said to be a \emph{ sequence of generic QMC designs} if \eqref{eq:approxfors} holds for all $s>d/2$.
\end{defn}

As in Definition~\ref{def:specific_approx.sph.design}, $X_N$ need not be defined for all natural numbers $N$.
% \MARKED{cyan}{Every sequence of spherical $t$-designs is a sequence of QMC designs for $\mathbb{H}^s( \mathbb{S}^d )$ for {\bf all} $s > d/2$ if and only if $N \asymp t^d$ as $t\to\infty$.}
Obviously, every sequence of spherical $t$-designs with $N \asymp t^d$ as $t\to\infty$ is a sequence of generic QMC designs for $\mathbb{H}^s( \mathbb{S}^d )$, for {\bf all} $s > d / 2$. We noted already the claimed existence of a sequence of spherical $t$-designs with $N \asymp t^d$.
% \MARKED{Orange}{\COMMENT{repeats p2l-5}
% The fact that sequences $(X_N)$ of spherical $t$-designs with $N\asymp t^d$ exist was proved recently by Bondarenko et.~al. \cite{BoRaVi2011arXiv}.} 
A simple application of that result yields the following.

\begin{thm} \label{thm:derrived.gen.sequ}
% Given $d \geq 2$, there exist $N$-point spherical $t$-designs $Y_N$ for $N = 1, 2, 3, \dots$ and $t \asymp N^{1/d}$ that form a sequence of generic QMC designs.
There exist $N$-point spherical $t$-designs $Y_N$ on $\mathbb{S}^d$ for $N = 1, 2, 3, \dots$ and $t \asymp N^{1/d}$ that form a sequence of generic QMC designs.
\end{thm}

For {\bf fixed} $s > d / 2$, there exist many sequences of QMC designs for $\mathbb{H}^s( \mathbb{S}^d )$ that are not composed of spherical designs.
Indeed, if $K$ is a reproducing kernel for the Sobolev space $\mathbb{H}^s( \mathbb{S}^d )$, $s > d/2$, we prove in Section \ref{sec:wce}
(by appealing to Theorem~\ref{thm:derrived.gen.sequ})
that $N$-point configurations ($N \geq 2$) minimizing the $K$-energy functional
\begin{equation}\label{eq:double}
\sum_{j=1}^N \sum_{i=1}^N K( \PT{x}_j, \PT{x}_i )
\end{equation}
form a sequence of QMC designs for this $\mathbb{H}^s( \mathbb{S}^d )$; cf. Theorem~\ref{thm:minimizers}.
For $s$ in the interval $(d/2, d/2+1)$ we show in Section~\ref{sec:generalized.distance.kernel} that, for $C$ a suitably large constant, $C - | \PT{x} - \PT{y} |^{2s-d}$ is a reproducing kernel for $\mathbb{H}^s( \mathbb{S}^d )$, and therefore 
the maximizers of the generalized sum of distances
\begin{equation} \label{eq:generalised.sum.dist}
\sum_{j=1}^N \sum_{i=1}^N \left| \PT{x}_j - \PT{x}_i \right|^{2s-d}, %, \qquad \text{$d/2 < s < d/2 + 1$ fixed,}
\end{equation}
$N = 2, 3, 4, \dots$ form a sequence of QMC designs for this $\mathbb{H}^s( \mathbb{S}^d )$. %, see Section~\ref{sec:generalized.distance.kernel}.

We also provide an alternative sufficient condition for QMC designs that utilizes polynomial truncations of a zonal reproducing kernel $K(\PT{x}, \PT{y}) = K(\PT{x}\cdot\PT{y})$, but requires also a regularity condition (Property R) imposed on the point configurations; see Definition~\ref{def:propertyR} and Theorem \ref{thm:characterization}. 

Numerical evidence presented later in this paper suggests that many familiar sequences of point sets on $\mathbb{S}^2$ (such as minimal [Coulomb or logarithmic] energy points, generalized spiral points, equal area points, and Fekete points) form sequences of QMC designs for $\mathbb{H}^s( \mathbb{S}^2 )$ for values of $s$ up to a supremum $s^*$ that depends on the particular sequence. Some conjectured values of $s^*$ are given in Section~\ref{sec:numerics}.

That the QMC design property \eqref{eq:approxfors} is not satisfied by all sequences of point sets follows from a probabilistic argument. 

\begin{thm}\label{thm:better}
% Given $d \geq 2$ and $s > d/2$, 
Given $s > d/2$, the expected value of the squared worst-case error satisfies
\begin{equation*}
\sqrt{\mathbb{E}\Bigg[ \sup_{\substack{f \in \mathbb{H}^s( \mathbb{S}^d ), \\ \| f \|_{\mathbb{H}^s} \leq 1}} \Bigg| \frac{1}{N} \sum_{j=1}^N f( \PT{x}_j ) - \int_{\mathbb{S}^d} f( \PT{x} ) \dd \sigma_d( \PT{x} ) \Bigg|^2 \Bigg]} = \frac{b(s,d)}{N^{1/2}}
\end{equation*}
for some constant $b(s,d) > 0$, where the points $\PT{x}_1, \dots, \PT{x}_N$ are  independently and uniformly distributed on $\mathbb{S}^d$.
\end{thm}

Theorem \ref{thm:better} tells us that randomly chosen point sets give a slower rate of convergence than $N^{-s/d}$ for all $s>d/2$, and hence do not form QMC designs. (See Section \ref{sec:better.than.average} for a more complete discussion.) 
However, if we compartmentalize the random point selection process with respect to a partition of the sphere into $N$ equal area regions with small diameter, then we do get an average worst-case error rate appropriate to QMC designs for $s \in (d/2, d/2+1)$, see Theorem~\ref{thm:equal.area}. On the other hand, such randomized equal area point configurations will, on average, not form a sequence of QMC designs for $s > d/2 + 1$, see Theorem~\ref{thm:equal.area.lower.bound}.  

We shall also discuss ``low-discrepancy sequences'' on the sphere and estimates for their worst-case error when used for QMC rules. It turns out that the point sets of such a sequence almost satisfy the QMC design property for $s \in (d/2, (d + 1)/ 2)$, except for a power of $\log N$.

\ 

The structure of the paper is as follows: The next section provides background for spherical designs and for Sobolev spaces $\mathbb{H}^s( \mathbb{S}^d)$ and their associated reproducing kernels. 
Section~\ref{sec:wce} characterizes the worst-case error for equal weight numerical integration and states two main results: Theorems~\ref{thm:minimizers} and \ref{thm:characterization}.
Examples of Sobolev spaces and associated kernels are given in Sections~\ref{sec:Cui.Freeden.kernel} and \ref{sec:generalized.distance.kernel}, with particular emphasis on configurations maximizing sums of generalized distances.
Section~\ref{sec:uniform.distribution} concerns low-discrepancy sequences on the sphere and their quadrature properties. 
In Section~\ref{sec:better.than.average} we analyze the quadrature error for randomly chosen points on the sphere and, in Section~\ref{sec:numerics}, we provide numerical results for worst-case errors and quadrature errors for certain familiar sequences of configurations.
Most of the formal proofs are given in Section~\ref{sec:proofs}.

\section{Background} \label{sec:prelim}

\subsection{Spherical designs}\label{sec:sd}
In the literature on spherical designs, and again in this paper, the relation between $N$ and $t$ in \eqref{eq:spherical.designs} plays an important role. It is known (Seymour and Zaslavsky~\cite{SeZa1984}), that a spherical $t$-design always exists if $N$ is sufficiently large  but that result says nothing about the size of $N$. In the important paper \cite{DeGoSei1977} lower bounds of exact order $t^d$ were established, precisely
\begin{equation}
\label{eq:DGS-bound}
N \geq
\begin{cases}
\displaystyle \binom{d+t/2}{d}+\binom{d+t/2-1}{d} & \text{for $t$ even,} \\[2ex]
\displaystyle 2\binom{d+\lfloor t/2\rfloor}{d} & \text{for $t$ odd,}
\end{cases}
\end{equation}
but it is known, see Bannai and Damerell~\cite{BaDa1979, BaDa1980}, that these lower bounds can be achieved only for a few small values of $t$.
Korevaar and Meyers~\cite{KoMe1993} conjectured that there always exist spherical $t$-designs with $N \asymp t^d$ points.
Bondarenko, Radchenko and Viazovska~\cite{BoRaVi2011arXiv} claim to have resolved this long-standing open problem, by establishing the following result.
% \footnote{The weaker result that there exist spherical $t$-designs with $|X|$ of order $t^{2d (d + 1) / (d + 2)}$ is proved in \cite{BoVi2010}.} 
\begin{thm} \label{thm:bondarenko.et.al}
For $d \geq 2$, there exists a constant $c_d$ depending only on $d$ such that for every $N \geq c_d \, t^d$ there exists a spherical $t$-design on $\mathbb{S}^d$ with $N$ points.
\end{thm}
Supporting evidence is provided in \cite{ChFrLa2011}, which used interval analysis to establish rigorously the existence of spherical designs  on $\mathbb{S}^2$ with $N=(t+1)^2$ for all values of $t\leq 100$.

\subsection{Spherical harmonics}\label{subsec:sphere}
Recall that $\sigma_d$ denotes the normalized (Hausdorff) surface measure on the unit sphere $\mathbb{S}^d$ in $\mathbb{R}^{d+1}$. % ($d \geq 2$).
The [non-normalized] surface area of $\mathbb{S}^d$ is denoted by $\omega_d$. For future reference, we record the following facts:
% For future reference we record (with $\omega_d$ denoting the [non-normalized] surface area of $\mathbb{S}^d$) the following facts:
\begin{equation} \label{eq:omega.d.ratio}
\gamma_d \DEF \frac{1}{d} \frac{\omega_{d-1}}{\omega_{d}}, \qquad \frac{\omega_{d-1}}{\omega_{d}} = \frac{\Gamma((d+1)/2)}{\sqrt{\pi} \Gamma(d/2)} \sim \frac{d^{1/2}}{\sqrt{2 \pi}} \quad \text{as $d \to \infty$.} % \frac{1}{2} \int_{\mathbb{S}^d} \left| \PT{p} \cdot \PT{y} \right| \dd \sigma_d( \PT{y}) =
\end{equation}
Here, $\Gamma(z)$ denotes the gamma function and $f(x) \sim g(x)$ as $x \to c$ means $f(x) / g(x) \to 1$ as $x \to c$. The asymptotic relation in \eqref{eq:omega.d.ratio} follows from (see \cite[Eq.~5.11.12]{DLMF2010.05.07})
\begin{equation} \label{eq:gamma.ratio}
\frac{\gammafcn(z + a)}{\gammafcn(z + b)} \sim z^{a-b} \qquad \text{as $z \to \infty$ in the sector $| \arg z | \leq \pi - \delta < \pi$.}
\end{equation}
We make use of the rising factorial, that is the {\em Pochhammer symbol} defined by
\begin{equation} \label{eq:Pochhammer.symb}
\Pochhsymb{a}{0} = 1, \qquad \Pochhsymb{a}{n+1} = \Pochhsymb{a}{n} (n + a), \quad n = 0, 1, \dots,
\end{equation}
which can be written in terms of the gamma function by means of $\Pochhsymb{a}{n} = \gammafcn(n+a) / \gammafcn(a)$.

We denote, as usual, by $\{ Y_{\ell,k} : k = 1, \dots, Z(d, \ell) \}$ a collection of $\mathbb{L}_2$-orthonormal
real spherical harmonics (homogeneous harmonic polynomials in $d+1$ variables restricted to $\mathbb{S}^d$) of exact degree $\ell$, where
\begin{equation} \label{eq:Z.d.n}
Z(d,n) = \left( 2n + d - 1 \right) \frac{\gammafcn(n+d-1)}{\gammafcn(d)\,\gammafcn(n+1)} \sim \frac{2}{\gammafcn(d)} \, n^{d-1} \quad \text{as $n \to \infty$}.
\end{equation}
It is well-known that the $Y_{\ell,k}$ satisfy the following identity known as the addition theorem:
\begin{equation} \label{eq:addition.theorem}
\sum_{k=1}^{Z(d,\ell)} Y_{\ell,k}( \PT{x} ) Y_{\ell,k}( \PT{y} ) = Z(d,\ell) \, P_\ell^{(d)}(\PT{x} \cdot \PT{y}), \qquad \PT{x}, \PT{y} \in \mathbb{S}^d,
\end{equation}
where $P_\ell^{(d)}$  is the normalized Gegenbauer (or Legendre) polynomial, orthogonal on the interval $[-1,1]$ with respect to the weight function $(1-t^2)^{d/2-1}$, and normalized by $P_\ell^{(d)}(1) = 1$.
Each spherical harmonic $Y_{\ell,k}$ of exact degree $\ell$ is an eigenfunction of the negative Laplace-Beltrami operator $-\Delta_d^*$ for $\mathbb{S}^d$ with eigenvalue
\begin{equation} \label{eq:eigenvalue}
% - \Delta_d^* \, Y_{\ell,k} = \lambda_\ell \; Y_{\ell,k}.
\lambda_\ell \DEF \ell \left( \ell + d - 1 \right).
\end{equation}
(For further details see, e.g., \cite{Mu1966}.)

The family $\{ Y_{\ell,k} : k = 1, \dots, Z(d,\ell); \ell = 0, 1, \dots \}$ forms a complete orthonormal (with respect to $\sigma_d$) system for the Hilbert space $\mathbb{L}_2(\mathbb{S}^{d})$ of square-integrable functions on $\mathbb{S}^d$ endowed with the usual inner product and induced norm
\begin{equation*}
( f, g )_{\mathbb{L}_2(\mathbb{S}^{d})} \DEF \int_{\mathbb{S}^d} f( \PT{x} ) g( \PT{x} ) \dd \sigma_d( \PT{x} ), \qquad \| f \|_{\mathbb{L}_2(\mathbb{S}^{d})} \DEF \sqrt{( f, f )_{\mathbb{L}_2(\mathbb{S}^{d})}}.
\end{equation*}

We shall denote by $\mathbb{P}_t(\Sp^d)$ the space of all spherical polynomials of degree $\leq t$ (that is, the restriction to $\mathbb{S}^d$ of all polynomials in $d+1$ real variables of degree $\leq t$). The space $\mathbb{P}_t(\mathbb{S}^d)$ coincides with the span of all spherical harmonics up to (and including) degree $t$, and its dimension is $Z(d+1,t)$.

We make frequent use of the following simple application of the addition theorem.

\begin{lem} \label{lem:positive.definiteness}
Let $d\geq 2$. For all integers $\ell \geq 0$ and all choices of points $\PT{x}_1, \dots, \PT{x}_N \in \mathbb{S}^d$ there holds
\begin{equation} \label{eq:positive.definiteness}
\Phi_\ell( \PT{x}_1, \dots, \PT{x}_N ) \DEF \frac{1}{N^2} \sum_{j=1}^N \sum_{i=1}^N Z( d, \ell ) \, P_\ell^{(d)}( \PT{x}_j \cdot \PT{x}_i ) = \sum_{k=1}^{Z( d, \ell )} \left| \frac{1}{N} \sum_{j=1}^N Y_{\ell,k} ( \PT{x}_j ) \right|^2 \geq 0.
\end{equation}
\end{lem}

\subsection{Sobolev spaces}\label{subsec:sobolev}

The Sobolev space $\mathbb{H}^s(\mathbb{S}^d)$ may be defined for $s\ge 0$ as the set of all functions $f\in \mathbb{L}_2(\mathbb{S}^d)$ whose Laplace-Fourier coefficients
\begin{equation} \label{eq:L-F.coeff}
\widehat{f}_{\ell,k} := ( f, Y_{\ell,k} )_{\mathbb{L}_2(\mathbb{S}^{d})} = \int_{\mathbb{S}^d} f( \PT{x} ) Y_{\ell,k}( \PT{x} ) \dd \sigma_d( \PT{x} )
\end{equation}
satisfy
\begin{equation}\label{eq:sobcond}
\sum_{\ell=0}^\infty \sum_{k=1}^{Z(d,\ell)} \left( 1 + \lambda_\ell \right)^{s} \left| \widehat{f}_{\ell,k} \right|^2 <\infty,
\end{equation}
where the $\lambda_\ell$'s are given in \eqref{eq:eigenvalue}.
On setting $s=0$ we recover $\mathbb{H}^0(\mathbb{S}^d)=\mathbb{L}_2(\mathbb{S}^d).$

The norm in  $\mathbb{H}^s(\mathbb{S}^d)$ may of course be defined as the square root of the expression on the left-hand side of the last inequality.
In this paper we shall, however, take advantage of the freedom to define equivalent Sobolev space norms. Let $s > d/2$ be fixed and suppose we are given a sequence of positive real numbers $(a_\ell^{(s)})_{\ell \geq 0}$  satisfying
\begin{equation} \label{eq:sequenceassumption}
a_\ell^{(s)} \asymp \left( 1 + \lambda_\ell \right)^{-s} \asymp \left( 1 + \ell \right)^{-2s}. % \qquad \text{for some }s>d/2.
\end{equation}
Then we can define a norm in $\mathbb{H}^s(\mathbb{S}^d)$ by
\begin{equation}\label{eq:sobnorm}
\|f\|_{\mathbb{H}^s} \DEF \left[\sum_{\ell=0}^\infty \sum_{k=1}^{Z(d,\ell)} \frac{1}{a_\ell^{(s)}} \left|\widehat{f}_{\ell,k}\right|^2\right]^{1/2}.
\end{equation}
The norm therefore depends on the particular choice of the sequence $(a_\ell^{(s)})_{\ell \geq 0}$, but for notational simplicity we shall generally not show this dependence explicitly.
Clearly, Definitions~\ref{def:specific_approx.sph.design} and \ref{def:approx.sph.design} are not tied to a particular Sobolev norm $\| \cdot \|_{\mathbb{H}^s}$, since a change to an {\em equivalent} Sobolev norm merely leads to a change of the constant $c(s,d)$. The corresponding inner product in the Sobolev space is
\begin{equation}\label{inner}
( f, g )_{\mathbb{H}^s} \DEF \sum_{\ell=0}^\infty \sum_{k=1}^{Z(d,\ell)} \frac{1}{a_\ell^{(s)}} \widehat{f}_{\ell,k} \, \widehat{g}_{\ell,k}.
\end{equation}

It is well known that $\mathbb{H}^{s}(\mathbb{S}^d) \subset \mathbb{H}^{s^\prime}(\mathbb{S}^d)$ whenever $s > s^\prime$, and that $\mathbb{H}^{s}(\mathbb{S}^d)$ is embedded in the space of $k$-times continuously differentiable functions $C^k(\mathbb{S}^d)$ if $s > k + d / 2$ (e.g. \cite{He2006}).

\subsection{Sobolev spaces as reproducing kernel Hilbert spaces} \label{subsec:H.as.RKHS}
Since the point-evaluation functional is bounded in $\mathbb{H}^s(\mathbb{S}^d)$ whenever $s>d/2$, the Riesz representation theorem assures the existence of a reproducing kernel $K^{(s)}(\PT{x},\PT{y})$, which can be written as
\begin{equation}\label{eq:K}
K^{(s)}(\PT{x},\PT{y}) = \sum_{\ell=0}^\infty a_\ell^{(s)} Z(d,\ell) P_\ell^{(d)}(\PT{x}\cdot\PT{y})=\sum_{\ell=0}^\infty \sum_{k=1}^{Z(d,\ell)}a_\ell^{(s)}Y_{\ell,k}(\PT{x})Y_{\ell,k}(\PT{y}),
\end{equation}
where the positive coefficients $a_\ell^{(s)}$ satisfy \eqref{eq:sequenceassumption}. 
It is easily verified that the above expression has the reproducing kernel properties
\begin{equation} \label{eq:repr.kernel.prop}
K^{(s)}({\PT \cdot},{\PT x}) \in \mathbb{H}^s(\mathbb{S}^d), \quad \PT{x} \in \mathbb{S}^d, \qquad ( f, K^{(s)}({\PT \cdot},{\PT x}) )_{\mathbb{H}^s} = f({\PT x}), \quad \PT{x} \in \mathbb{S}^d, f \in \mathbb{H}^s(\mathbb{S}^d).
\end{equation}
The kernel is a {\em zonal} function; that is $K^{(s)}(\PT{x},\PT{y})$ depends only on the inner product $\PT{x}\cdot\PT{y}$.  We write for simplicity $K^{(s)}( \PT{x} \cdot \PT{y} ) \DEF K^{(s)}( \PT{x}, \PT{y} )$.
For the particular choice $a_\ell^{(s)} = ( 1 + \lambda_\ell )^{-s}$, we use the notation 
\begin{equation}\label{eq:K.can}
K_{\mathrm{can}}^{(s)}(\PT{x},\PT{y}) \DEF \sum_{\ell=0}^\infty ( 1 + \lambda_\ell )^{-s} Z(d,\ell) P_\ell^{(d)}(\PT{x}\cdot\PT{y}),
\end{equation}
which we call the \emph{canonical kernel} for $\mathbb{H}^{s}( \mathbb{S}^d )$.

Sections~\ref{sec:Cui.Freeden.kernel} and \ref{sec:generalized.distance.kernel} contain explicit examples of Sobolev spaces and associated kernels.

\section{Numerical integration and worst-case error}
\label{sec:wce}

\subsection{Worst-case error}

Our results are based on an explicit expression for the ``worst-case error'' that occurs on the left-hand side of \eqref{eq:approxfors}:
\begin{defn} \label{def:wce}
For a Banach space $B$ of continuous functions on $\mathbb{S}^d$ with norm $\|\cdot\|_B$, the {\em worst-case error} for the integration rule $\numint[X_N]$ with node set $X_N = \{ \PT{x}_1,\ldots,\PT{x}_N \}$ approximating the integral $\xctint(f)$, with $\numint[X_N](f)$ and $\xctint(f)$ defined by
\begin{equation} \label{eq:QN}
\numint[X_N](f) \DEF \frac{1}{N} \sum_{j=1}^N f(\PT{x}_j), \qquad \xctint(f) \DEF \int_{\mathbb{S}^d} f( \PT{x} ) \dd \sigma_d( \PT{x} ),
\end{equation}
is given by%\footnote{The supremum can be extended over functions $f\in B$,  }
\begin{equation}\label{eq:wceB}
\wce(\numint[X_N]; B ) \DEF \sup \left\{ \big| \numint[X_N](f) - \xctint(f) \big| : f \in B, \| f \|_B \leq 1 \right\}.
\end{equation}
\end{defn}

As a trivial consequence of the definition we have the following error bound for an arbitrary function  $f\in B$:
\begin{equation} \label{eq:Koksma.like}
\left| \numint[X_N](f) - I(f) \right| \leq \wce(\numint[X_N]; B ) \left\|f \right\|_B.
\end{equation}
Because of the similarity of \eqref{eq:Koksma.like} to the celebrated Koksma-Hlawka inequality, which involves the ``star-discrepancy'' of the node set, 
the worst-case error is sometimes referred to as a {\em generalized discrepancy}, see for example \cite{CuFr1997}.  In this paper, however, we shall generally reserve the word ``discrepancy'' for quantities that have a geometric interpretation.

\subsection{Worst-case error in a reproducing kernel Hilbert space} \label{subsec:wce.rkhs}

For most reproducing kernel Hilbert spaces there is a computable expression for the worst-case error, as shown by the following standard argument.  With $K$  the kernel of a reproducing kernel Hilbert space $H$ with inner product $(\cdot,\cdot)_H$,
the reproducing kernel property $f(\PT{x})=(f,K(\cdot,\PT{x}))_H$ allows us to write
\begin{equation*}
\numint[X_N](f) - \xctint(f) = ( f, \mathcal{R}[X_N] )_H, \qquad f \in H,
\end{equation*}
where $\mathcal{R}[X_N] \in H$ is the ``representer'' of the error, given by
\begin{equation*}
\mathcal{R}[X_N]( \PT{x} ) \DEF \frac{1}{N} \sum_{j = 1}^{N} \, K(\PT{x}, \PT{x}_j) - I_{\PT{y}} K(\PT{x},\cdot) ,
\end{equation*}
assuming that the integration functional $f \mapsto \xctint(f)$ is bounded on $H$.
Here $I_{\PT{y}} K$ means the integral functional $I$ applied to the second variable in $K$ (and later $I_{\PT{x}} K$ will mean the integral functional applied to the first variable).\footnote{The norm $\| \mathcal{R}[X_N] \|_H$ is also known as the \emph{$g$-diaphony of $X_N$} with $g = I_{\PT{y}} K(\PT{x},\cdot)$; see \cite{AmZi2001}.}

It follows that
\begin{align*}
\left[ \wce( \numint[X_{N}]; H ) \right]^2
&\DEF \Big[ \sup \left\{ \big| \numint[X_N](f) - \xctint(f) \big| : f \in H, \| f \|_H \leq 1 \right\} \Big]^2 \\
&= \Big[ \sup \left\{ \big| ( f, \mathcal{R}[X_N] )_H \big| : f \in H, \| f \|_H \leq 1 \right\} \Big]^2 \\
&= \left\| \mathcal{R}[X_N] \right\|_H^2 = \left( \mathcal{R}[X_N], \mathcal{R}[X_N] \right)_H \\
&= \frac{1}{N^2}\sum_{j=1}^{N} \sum_{i=1}^{N} \, K({\PT x}_{j},{\PT x}_{i}) - \frac{2}{N} \sum_{j=1}^{N} I_{\PT{y}} K({\PT x}_{j},\cdot) + I_{\PT{x}} I_{\PT{y}} K(\cdot,\cdot).
\end{align*}

\subsection{Worst-case error in $\mathbb{H}^s( \mathbb{S}^d )$} %for $s>d/2$}

Now consider the special case of the reproducing kernel Hilbert space $\mathbb{H}^s( \mathbb{S}^d )$ with $s>d/2$, and with reproducing kernel given by \eqref{eq:K}.
For this case it is easily seen that
\begin{equation*}
I_{\PT{y}} K^{(s)}(\PT{x},\cdot) = a_0^{(s)},
\end{equation*}
from which it follows that
\begin{equation} \label{eq:wce2}
\left[ \wce( \numint[X_{N}]; \mathbb{H}^s( \mathbb{S}^d ) ) \right]^2 = \left[ \frac{1}{N^2}\sum_{j=1}^{N} \sum_{i=1}^{N} \, K^{(s)}({\PT x}_j,{\PT x}_i) \right] - a_0^{(s)} = \frac{1}{N^2}\sum_{j=1}^{N} \sum_{i=1}^{N} \, \mathcal{K}^{(s)}({\PT x}_j \cdot {\PT x}_i),
\end{equation}
where $\mathcal{K}^{(s)}:[-1,1]\to \mathbb{R}$ is defined  by
\begin{equation}\label{calK}
\mathcal{K}^{(s)}(\PT{x} \cdot \PT{y}) \DEF \sum_{\ell=1}^\infty a_\ell^{(s)} Z(d,\ell) P_\ell^{(d)}(\PT{x} \cdot \PT{y}). %\qquad \PT{x}, \PT{y} \in \mathbb{S}^d,
\end{equation}
The use of the calligraphic symbol here and for subsequent kernels indicates that the sum runs from $\ell=1$ rather than $\ell=0$. 
Note that the kernels depend on $s$ through the sequence $(a_\ell^{(s)})_{\ell\ge 0}$. 

We summarize these observations in the following proposition.

\begin{prop} \label{prop:wce}
For $s>d/2$, let $\mathbb{H}^s(\mathbb{S}^d)$ be the Hilbert space with norm \eqref{eq:sobnorm}, where the sequence $(a_\ell^{(s)})_{\ell\ge0}$ satisfies \eqref{eq:sequenceassumption}, and let $\mathcal{K}^{(s)}$ be given by \eqref{calK}.  Then, for a rule $\numint[X_N]$ with node set $X_N = \{ \PT{x}_1, \dots, \PT{x}_N \} \subset \mathbb{S}^d$,
\begin{equation}\label{eq:earlywce}
\begin{split}
\wce( \numint[X_N]; \mathbb{H}^s(\mathbb{S}^d) ) 
&= \left( \frac{1}{N^2} \sum_{j=1}^N \sum_{i=1}^N \, \mathcal{K}^{(s)}({\PT x}_j \cdot {\PT x}_i ) \right)^{1/2} \\
&= \left(\sum_{\ell=1}^\infty \sum_{k=1}^{Z(d,\ell)}a_\ell^{(s)} \left|\frac{1}{N}\sum_{j=1}^N Y_{\ell,k}({\PT x}_j)\right|^2\right)^{1/2}.
\end{split}
\end{equation}
\end{prop}

From the first expression in \eqref{eq:earlywce}, the squared worst-case error for the rule $\numint[X_N]$ is the normalized $\mathcal{K}^{(s)}$-energy functional evaluated at the node set $X_N$. This expression can be computationally useful when the kernel $\mathcal{K}^{(s)}$ is available in closed form.

By comparison with a sequence satisfying Theorem~\ref{thm:derrived.gen.sequ}, we deduce that the minimizers of \eqref{eq:earlywce} yield a sequence of QMC designs for $\mathbb{H}^s( \mathbb{S}^d )$.

\begin{thm} \label{thm:minimizers}
Under the assumptions of Proposition~\ref{prop:wce}, if 
% $X_N^* = \{ \PT{x}_{1,N}^*, \dots, \PT{x}_{N,N}^* \}$, 
$X_N^*$, $N = 2, 3, 4, \dots$, minimizes the energy functional
\begin{equation*}
\sum_{j=1}^N \sum_{i=1}^N \, \mathcal{K}^{(s)}({\PT x}_j \cdot {\PT x}_i ),% \qquad \PT{x}_1, \dots, \PT{x}_N \in \mathbb{S}^d,
\end{equation*}
then there exists $c(s,d) > 0$ depending on the $\mathbb{H}^s( \mathbb{S}^d )$-norm such that for all $N \geq 2$
\begin{equation*} %\label{eq:minimizer.rel}
\wce( \numint[X_N^*]; \mathbb{H}^s( \mathbb{S}^d ) )
\leq \frac{c(s,d)}{N^{s/d}}.
\end{equation*}
Consequently, $(X_N^*)$ is a sequence of QMC designs for $\mathbb{H}^s( \mathbb{S}^d )$.
\end{thm}

The next result says, in essence, that the computed worst-case errors in $\mathbb{H}^s(\mathbb{S}^d )$ of a given sequence of QMC designs for $\mathbb{H}^{s'}(\mathbb{S}^d)$, where $d/2 < s^\prime \leq s$, will always show a rate of decay at least $O(N^{-s^\prime/d})$ and at most $O(N^{-s/d})$. 

% \begin{thm} \label{thm:best.rate}
% For $s>d/2$, let $\mathbb{H}^s(\mathbb{S}^d)$ be the Hilbert space with norm \eqref{eq:sobnorm}, where the sequence $(a_\ell^{(s)})_{\ell\ge0}$ satisfies \eqref{eq:sequenceassumption}, and let $\mathcal{K}^{(s)}$ be given by \eqref{calK}.  Moreover, let $(X_N)$ be a sequence of QMC designs for $\mathbb{H}^{s'}(\mathbb{S}^d)$, where $d/2<s'\le s$.  Then there exists $c(s,s',d)>0$ such that 
% \begin{equation}
% \wce(\numint[X_N]; \mathbb{H}^s(\mathbb{S}^d) )
% % \leq \wce( \numint[Y_N]; \mathbb{H}^s( \mathbb{S}^d ) )
% \leq \frac{c(s,s',d)}{N^{s'/d}}.
% \end{equation}
% \end{thm}
\begin{thm} \label{thm:best.rate}
Under the assumptions of Proposition~\ref{prop:wce} for $\mathbb{H}^s( \mathbb{S}^d )$ and $\mathbb{H}^{s^\prime}( \mathbb{S}^d )$, if $(X_N)$ is a sequence of QMC designs for $\mathbb{H}^{s'}(\mathbb{S}^d)$, then 
% there exists $c(s, s^\prime, d) > 0$ depending on the $\mathbb{H}^s( \mathbb{S}^d )$-norm and the $\mathbb{H}^{s^\prime}( \mathbb{S}^d )$-norm such that
\begin{equation}
\wce(\numint[X_N]; \mathbb{H}^s(\mathbb{S}^d) )
% \leq \wce( \numint[Y_N]; \mathbb{H}^s( \mathbb{S}^d ) )
\leq \frac{c(s,s',d)}{N^{s'/d}}, \qquad d/2 < s^\prime < s,
\end{equation}
where $c(s, s^\prime, d) > 0$ depends on the norms for $\mathbb{H}^s( \mathbb{S}^d )$ and $\mathbb{H}^{s^\prime}( \mathbb{S}^d )$, but is independent of~$N$.
\end{thm}

This result follows from the last expression in \eqref{eq:earlywce}, since it implies that, with respect to the canonical kernels \eqref{eq:K.can} for $\mathbb{H}^s( \mathbb{S}^d )$ and $\mathbb{H}^{s^\prime}( \mathbb{S}^d )$, there holds for any $N$-point configuration $X_N$, 
\begin{equation} \label{eq:wce.inequality.A}
\wce( \numint[X_N]; \mathbb{H}^s( \mathbb{S}^d ) ) < \wce( \numint[X_N]; \mathbb{H}^{s^\prime}( \mathbb{S}^d ) ), \qquad \text{$d/2 < s^\prime < s$.}
\end{equation}

We shall exploit Theorem~\ref{thm:best.rate} in Section~\ref{sec:numerics} to determine empirical values of the supremum $s^*$ in \eqref{eq:s.star} for a number of putative sequences of QMC designs.

\subsection{Truncations of the Laplace-Fourier series} \label{subsec:truncation}

Using truncations of the Laplace-Fourier series of a fixed reproducing kernel, we provide in the next theorem a sufficient condition for a sequence of point sets to be a QMC design sequence. 
It relies on the following regularity condition, first introduced by Sloan and Womersley in \cite{SlWo2000} in the context of positive weight quadrature rules. This imposed condition allows us to control the contribution to the worst-case error of the suppressed tail-part of the aforementioned kernel.

\begin{defn}[Property R] \label{def:propertyR}
A sequence $(Z_N)$ of $N$-point configurations on $\mathbb{S}^d$ is said to have the Property R if there exist positive numbers $c_0$ and $c_1$ independent of $N$,  such that for all $Z_N$ and all $\PT{z} \in \mathbb{S}^d$ the nodes satisfy
\begin{equation*} %\label{reg-prop}
|Z_N\cap \mathcal{C}(\PT{z};c_1 N^{-1/d})| \leq c_0,
\end{equation*}
where 
\begin{equation} \label{eq:spherical.cap}
\mathcal{C}(\PT{z}; \theta) \DEF \{ \PT{y} \in \mathbb{S}^d : \PT{y} \cdot \PT{z} \geq \cos \theta \}
\end{equation}
denotes the spherical cap of geodesic radius $\theta$ centered at $\PT{z}$.
\end{defn}
The regularity Property R expresses a natural requirement, that the number of points in a spherical cap with radius of order $N^{-1/d}$, and hence whose area is of order $N^{-1}$, should be bounded independently of both $N$ and the location of the cap. 

\begin{rmk}
A sequence of $N$-point sets $(X_N)$, $X_N = \{ \PT{x}_{1,N}, \dots, \PT{x}_{N,N} \}$, is called \emph{well-separated} if there exists a constant $c_d > 0$ such that
\begin{equation*}
\min_{j \neq i} \left| \PT{x}_{j,N} - \PT{x}_{i,N} \right| \geq \frac{c_d}{N^{1/d}}, \qquad N = 2, 3, 4, \dots.
\end{equation*}
It is easily seen that a well-separated sequence of configurations satisfies Property R, but not conversely.
Furthermore, Reimer~\cite{Re2000} has shown that every sequence of spherical $t$-designs with $N \asymp t^d$ points automatically satisfies Property R.
\end{rmk}

\begin{thm} \label{thm:characterization}
% Let $s > d/2 \geq 1$ and let $\mathbb{H}^s( \mathbb{S}^d )$ be provided with a reproducing kernel \eqref{eq:K} defined by a sequence of positive numbers $(a_\ell^{(s)})_{\ell \geq 0}$ satisfying \eqref{eq:sequenceassumption}. Then 
Under the assumptions of Proposition~\ref{prop:wce}, a sequence $(X_N)$ of $N$-point configurations on $\mathbb{S}^d$ satisfying Property R is a sequence of QMC designs for $\mathbb{H}^s(\mathbb{S}^d)$ if and only if for integers $t$ with $t \asymp N^{1/d}$ %$N \asymp t^d$
\begin{equation} \label{eq:char.rel}
\frac{1}{N^2} \sum_{j = 1}^N \sum_{i = 1}^N \mathcal{K}_t^{(s)}( \PT{x}_{j,N} \cdot \PT{x}_{i,N} ) = \mathcal{O}(N^{-2s/d}) \qquad \text{as $N \to \infty$,}
\end{equation}
where $\mathcal{K}_t^{(s)}$ is the truncated kernel corresponding to \eqref{calK},
\begin{equation} \label{eq:K.t}
\mathcal{K}_t^{(s)}(\PT{x} \cdot \PT{y}) \DEF \sum_{\ell=1}^t a_\ell^{(s)} Z(d,\ell) P_\ell^{(d)}(\PT{x} \cdot \PT{y}).
\end{equation}
\end{thm}

This fits well with the framework of variational characterization of spherical designs (see \cite{ChWo2006, CoKu2007, GrTi1993, SlWo2009}), where one seeks point configurations for which the left-hand side in \eqref{eq:char.rel} vanishes. %(see \cite{AnChSlWo2010, ChWo2006, GrTi1993, SlWo2009})
The proof of Theorem~\ref{thm:characterization} relies on a non-trivial ``kernel-splitting'' argument that is discussed in Section~\ref{sec:proofs}.

\section{Cui and Freeden kernel}
\label{sec:Cui.Freeden.kernel}

For $\mathbb{S}^2$, Cui and Freeden \cite{CuFr1997} studied the kernel
\begin{equation} \label{eq:K.CF}
K_\CF(\PT{x}, \PT{y}) \DEF 1 + \sum_{\ell=1}^\infty \frac{1}{\ell ( \ell + 1 )} \, P_\ell( \PT{x} \cdot \PT{y} ) = 2 - 2 \log\Big( 1 + \sqrt{ \frac{1-\PT{x}\cdot\PT{y}}{2}} \Big),\qquad \PT{x},\PT{y}\in\mathbb{S}^2.
\end{equation}
It was observed in \cite{SlWo2004} that this is a reproducing kernel for $\mathbb{H}^{3/2}( \mathbb{S}^2 )$ as can be seen from the above Laplace-Fourier series expansion in terms of Legendre polynomials; cf. Section~\ref{sec:prelim}. 
Since the constant term in the series expansion \eqref{eq:K.CF} is $1$, Proposition~\ref{prop:wce} asserts that the corresponding worst-case error for $\mathbb{H}^{3/2}( \mathbb{S}^2 )$ equipped with this kernel is 
\begin{equation} \label{eq:Cui.Freeden.discrepancy}
\wce( \numint[X_N]; \mathbb{H}^{3/2}(\mathbb{S}^2) ) = \left(1 - \frac{2}{N^2} \sum_{j=1}^N \sum_{i=1}^N \log\Big( 1 + \sqrt{\frac{1 - {\PT x}_j \cdot {\PT x}_i}{2}} \Big) \right)^{1/2}.
\end{equation}

The right-hand side, up to a constant factor, is known as the \emph{Cui and Freeden (CF) discrepancy} of $X_N$. 
Note that by Theorem~\ref{thm:minimizers}, \emph{sequences of $N$-point configurations that minimize the CF discrepancy are sequences of QMC designs for $\mathbb{H}^{3/2}( \mathbb{S}^2 )$}. This fact can also be seen without appealing to results for spherical designs by applying the independently derived Theorem~\ref{thm:firstmax} in the next section and using the equivalence of norms.\footnote{An even more direct proof can be given by applying \cite[Theorem~2.2]{RaSaZh1994}.}
In \cite{CuFr1997} this discrepancy has been used to test for uniformity of a sequence of $N$-point configurations.
Furthermore, the CF discrepancy was used in \cite{SlWo2004} for analyzing quadrature properties of so-called extremal or Fekete points (these are sets of $(t+1)^2$ points on $\mathbb{S}^2$ that maximize the determinant of the interpolation matrix for polynomials of degree $t$).
Numerical data in \cite{SlWo2004} suggested that for $\mathbb{H}^{3/2}( \mathbb{S}^2 )$, the CF discrepancy for spherical $t$-designs obtained using Fekete points as starting points, decays like $\mathcal{O}(t^{-3/2})$; this in turn led to the discovery of Theorem \ref{thm:s.d.prop}.
% , a motivating factor for the concept of QMC designs. 

\section{Generalized distance kernel}
\label{sec:generalized.distance.kernel}

In the following we make use of the identity
\begin{equation*}
\left| \PT{x} - \PT{y} \right|^2 = 2 - 2 \PT{x} \cdot \PT{y}, \qquad \PT{x}, \PT{y} \in \mathbb{S}^d.
\end{equation*}

Reproducing kernels for $\mathbb{H}^s( \mathbb{S}^d )$ for $s > d/2$ can be constructed utilizing powers of distances, provided the power $2s-d$ is not an even integer. Indeed, it is known (cf., e.g., \cite{HuBa2001}) that the signed power of the distance, with sign $(-1)^{L+1}$ with $L \DEF L(s) \DEF \lfloor s - d/2 \rfloor$, has the following Laplace-Fourier expansion: 
\begin{equation} \label{eq:signed.distance}
(-1)^{L+1} \left| \PT{x} - \PT{y} \right|^{2s-d} = (-1)^{L+1} V_{d-2s}( \mathbb{S}^d ) + \sum_{\ell=1}^\infty \alpha_\ell^{(s)} \, Z(d, \ell) \, P_\ell^{(d)}( \PT{x} \cdot \PT{y} ), \qquad \text{$\PT{x}, \PT{y} \in \mathbb{S}^d$,}
\end{equation}
where
\begin{equation} \label{eq:gd.kernel.coeffs}
\begin{split}
V_{d-2s}( \mathbb{S}^d ) 
&\DEF \int_{\mathbb{S}^d} \int_{\mathbb{S}^d} \left| \PT{x} - \PT{y} \right|^{2s-d} \dd \sigma_d( \PT{x} ) \dd \sigma_d( \PT{y} ) = 2^{2s-1} \frac{\gammafcn((d+1)/2) \gammafcn(s)}{\sqrt{\pi} \gammafcn(d/2+s)}, \\
\alpha_\ell^{(s)} 
&\DEF V_{d-2s}( \mathbb{S}^d ) \, \frac{(-1)^{L+1} \Pochhsymb{d/2-s}{\ell}}{\Pochhsymb{d/2+s}{\ell}}, \quad \ell \geq 1.
\end{split}
\end{equation}
From these formulas one can verify that all the coefficients $\alpha_\ell^{(s)}$ are positive for $\ell \geq L+1$ and alternate in sign for $\ell \leq L$. 
Furthermore, the $\alpha_\ell^{(s)}$'s decay with the rate required for coefficients in the Laplace-Fourier expansion of a reproducing kernel for $\mathbb{H}^s( \mathbb{S}^d )$ as can be seen from the asymptotic expansion
\begin{equation} \label{eq:alpha.asymptotics}
\alpha_\ell^{(s)} \sim 2^{2s-1} \frac{\gammafcn((d+1)/2) \gammafcn(s)}{\sqrt{\pi} [ (-1)^{L+1} \gammafcn(d/2-s) ]} \, \ell^{-2s} \qquad \text{as $\ell \to \infty$.}
\end{equation}
Thus, by modifying if necessary some of the early coefficients, one can derive a reproducing kernel for $\mathbb{H}^s( \mathbb{S}^d )$ for $s > d/2$ and $2s-d$ not an even integer\footnote{In the case of $2s-d$ is an even integer, the expansion \eqref{eq:signed.distance} terminates after finitely many terms and so the $\alpha_\ell^{(s)}$'s do not satisfy the appropriate asymptotic behavior \eqref{eq:sequenceassumption} for $\mathbb{H}^s( \mathbb{S}^d )$.} (cf. Section~\ref{subsec:H.as.RKHS}). 

\subsection*{Case I}

For $d/2 < s < d/2 + 1$ (in which case, $L(s) = 0$), only the constant term in \eqref{eq:signed.distance} is negative and thus by adding any constant larger than $V_{d-2s}( \mathbb{S}^d )$, say $2V_{d-2s}( \mathbb{S}^d )$, we obtain the following reproducing kernel for $\mathbb{H}^s( \mathbb{S}^d )$ which we call the  ``generalized distance'' kernel:
\begin{equation} \label{eq:K.gd}
K_{\rm{gd}}^{(s)}( \PT{x}, \PT{y} ) \DEF 2 V_{d-2s}( \mathbb{S}^d ) - \left| \PT{x} - \PT{y} \right|^{2s-d}, \qquad \PT{x}, \PT{y} \in \mathbb{S}^d.
\end{equation} 
In particular, for $s=(d+1)/2$ we get the ``distance kernel'' for $\mathbb{H}^{(d+1)/2}( \mathbb{S}^d )$:
\begin{equation} \label{eq:K.C.s}
K_{\mathrm{dist}}( \PT{x} , \PT{y} ) \DEF K_{\rm{gd}}^{((d+1)/2)}( \PT{x}, \PT{y} ) = 2 V_{-1}(\mathbb{S}^d) - \left| \PT{x} - \PT{y} \right|, \qquad \PT{x}, \PT{y} \in \mathbb{S}^d,
\end{equation}
which for $d=2$ is equivalent to the Cui and Freeden kernel in the sense that there exist positive constants $c$ and $C$ such that $c \, K_{\mathrm{dist}}( \PT{x}, \PT{y} ) \leq K_\CF( \PT{x}, \PT{y} ) \leq C \, K_{\mathrm{dist}}( \PT{x}, \PT{y} )$ for all $\PT{x}, \PT{y} \in \mathbb{S}^2$.

With respect to the $K_{\rm{gd}}^{(s)}$ kernel, the worst-case error for $d/2 < s < d/2+1$ is, from Proposition~\ref{prop:wce}, given by
\begin{equation} \label{eq:dist.kernel.wce.formula}
\wce( \numint[X_N]; \mathbb{H}^{s}(\mathbb{S}^d) ) = \left( V_{d-2s}(\mathbb{S}^d) - \frac{1}{N^2} \sum_{j=1}^N \sum_{i=1}^N \left| \PT{x}_j - \PT{x}_i \right|^{2s-d} \right)^{1/2}.
\end{equation}
According to Theorem~\ref{thm:minimizers}, for $d/2 < s < d/2 + 1$, minimizing the right-hand side above, or equivalently, maximizing the sum of generalized distances, yields QMC designs for $\mathbb{H}^s( \mathbb{S}^d )$. This fact can also be established without appealing to any properties of spherical designs (and hence is independent of Theorem~\ref{thm:bondarenko.et.al}).
Indeed, Wagner~\cite{Wa1992}, extending a result of Stolarsky~\cite{St1973}, showed that for $d/2 < s < d/2 + 1$ there exists a sequence of $N$-point configurations $\{ \PT{x}_{1,N}^*, \dots, \PT{x}_{N,N}^* \}$ and a positive constant $\eta_{s,d}$ such that 
\begin{equation} \label{WAGNERsbounds}
% \eta_{s,d} \, N^{-2s/d} \leq 
V_{d-2s}( \mathbb{S}^d ) - \frac{1}{N^{2}} \sum_{j = 1}^N \sum_{i = 1}^N \left| \PT{x}_{j,N}^* - \PT{x}_{i,N}^* \right|^{2s-d} \leq \eta_{s,d} \, N^{-2s/d}, \qquad N \geq 2.
\end{equation}
(This fact also follows immediately from Theorem~\ref{thm:equal.area} below dealing with randomized equal area points on $\mathbb{S}^d$.)
% 
% Results of this type for more general kernels can be obtained from \cite[Theorem~2.2]{RaSaZh1994}. 
Consequently, we have provided an independent proof of the following result.

% \begin{thm} \label{thm:firstmax}
% Let $d \geq 2$ and $s \in (d/2, d/2+1)$. Suppose $(X_N^{(s)})$ is a sequence of $N$-point configurations $X_N^{(s)}$ maximizing the generalized sum of  Euclidean distances \eqref{eq:generalised.sum.dist}. Then $(X_N^*)$ is a sequence of QMC designs for $\mathbb{H}^s( \mathbb{S}^d )$.
% \end{thm}
\begin{thm} \label{thm:firstmax}
Given $s \in (d/2, d/2+1)$, a sequence of $N$-point sets $X_N^*$ that maximize the generalized sum of  Euclidean distances \eqref{eq:generalised.sum.dist} is a sequence of QMC designs for $\mathbb{H}^s( \mathbb{S}^d )$.
\end{thm}
 
We remark that for $s > d/2 + 1$, $N$-point configurations on $\mathbb{S}^d$ with maximum generalized sum of distances (without further restrictions) will have a limit distribution that is concentrated in two opposite points on $\mathbb{S}^d$ (Bj{\"o}rck~\cite[Remark~1 following Theorem~7]{Bj1956}) and, clearly, do not lead to  QMC designs. 

\subsection*{Case II}

For $s > d/2 + 1$ and $L$ as defined as above (so that $d/2 + L < s < d/2 + L + 1$), the representation~\eqref{eq:signed.distance} gives rise to a reproducing kernel for $\mathbb{H}^s( \mathbb{S}^d )$ of the form 
\begin{equation} \label{eq:K.gd.general}
K^{(s)}_{\rm{gd}}( \PT{x}, \PT{y} ) \DEF 
\left( 1 - (-1)^{L+1} \right)  V_{d-2s}( \mathbb{S}^d ) + \mathcal{Q}_L( \PT{x} \cdot \PT{y} ) + (-1)^{L+1} \left| \PT{x} - \PT{y} \right|^{2s-d}, \qquad \PT{x}, \PT{y} \in \mathbb{S}^d,
\end{equation}
where $\mathcal{Q}_L$ is a polynomial of degree $L \geq 1$, 
\begin{equation} \label{eq:cal.Q.L}
\mathcal{Q}_L( \PT{x} \cdot \PT{y} ) \DEF \sum_{\ell=1}^L \left( (-1)^{L+1-\ell} - 1 \right) \alpha_\ell^{(s)} \, Z(d, \ell) \, P_\ell^{(d)}( \PT{x} \cdot \PT{y} ), \qquad \PT{x}, \PT{y} \in \mathbb{S}^d,
\end{equation}
that simply changes the signs of the negative coefficients $\alpha_\ell^{(s)}$, $\ell \geq 1$, in \eqref{eq:signed.distance}.
As a consequence of Theorem~\ref{thm:minimizers}, we obtain the following.
\begin{thm}
Given $s \in (d/2+L, d/2+L+1)$, where $L$ is a positive integer, the sequences of $N$-point sets $X_N^{*}$ that minimize the worst-case error
\begin{equation} \label{eq:wce.4.K.dg.kernel}
\begin{split}
&\wce( \numint[X_N]; \mathbb{H}^{s}(\mathbb{S}^d) ) \\
&\phantom{equal}= \left( \frac{1}{N^2} \sum_{j=1}^N \sum_{i=1}^N \left[ \mathcal{Q}_L( \PT{x}_j \cdot \PT{x}_i ) + (-1)^{L+1} \left| \PT{x}_j - \PT{x}_i \right|^{2s-d} \right] - (-1)^{L+1} V_{d-2s}( \mathbb{S}^d ) \right)^{1/2}
\end{split}
\end{equation}
form a sequence of QMC designs for $\mathbb{H}^s( \mathbb{S}^d )$. 
\end{thm}

For fixed $L \geq1$, one can avoid the introduction of the $\mathcal{Q}_L$-energy term in the worst-case error formula above by restricting attention to node sets $X_{N,L} = \{ \PT{x}_{1,L}, \dots, \PT{x}_{N,L} \}$ that are spherical $L$-designs; i.e., satisfy %\eqref{eq:side.condition}.
\begin{equation} \label{eq:side.condition}
\sum_{j=1}^N Y_{\ell,k}( \PT{x}_{j,L} ) = 0, \qquad 1 \leq \ell \leq L, \, 1 \leq k \leq Z(d,\ell).
\end{equation} 
For such sequences the worst-case error formula reduces to
\begin{equation*}
\left[ \wce( \numint[X_{N,L}]; \mathbb{H}^s( \mathbb{S}^d ) ) \right]^2 = \frac{1}{N^2} \sum_{j=1}^{N} \sum_{i=1}^{N} (-1)^{L+1} \left| \PT{x}_{j,L} - \PT{x}_{i,L} \right|^{2s-d} - (-1)^{L+1} V_{d-2s}(\mathbb{S}^d).
\end{equation*}
Thus, spherical $L$-design configurations that minimize
\begin{equation*}
\sum_{j=1}^{N} \sum_{i=1}^{N} (-1)^{L+1} \left| \PT{x}_{j,L} - \PT{x}_{i,L} \right|^{2s-d}, \qquad N \geq c_d \, L^d,
\end{equation*}
yield sequences of QMC designs for $\mathbb{H}^s( \mathbb{S}^d )$ whenever $2s-d$ is not an even integer. 

Note that for even $2s - d = 2L = 2, 4, \dots$, the expansion \eqref{eq:signed.distance} terminates and is a polynomial of degree $L$ in $\PT{x} \cdot \PT{y}$. In such a case the radial (signed) generalized distance $(-1)^{L+1} | \PT{x} - \PT{y} |^{2L} \log | \PT{x} - \PT{y} |$ can be used to define a reproducing kernel for $\mathbb{H}^{2L}(\mathbb{S}^d)$. This approach is explored in \cite{BrWo2010Manuscript}.

\section{Uniform distribution and low-discrepancy sequences on the sphere}
\label{sec:uniform.distribution}

\subsection*{Uniform distribution}

An infinite sequence $(X_N)$ of $N$-point configurations on $\mathbb{S}^d$ is {\em asymptotically uniformly distributed on $\mathbb{S}^d$} if for every $f \in C( \mathbb{S}^d )$ there holds
\begin{equation} \label{eq:weak-star.convergence}
\numint[X_N](f) \to \int_{\mathbb{S}^d} f( \PT{x} ) \dd \sigma_d( \PT{x} ) \quad \text{as $N \to \infty$.}
\end{equation}

For a sequence $(X_N)$ of QMC designs for $\mathbb{H}^s( \mathbb{S}^d )$, $s > d/2$, it follows from \eqref{eq:Koksma.like} and \eqref{eq:sequenceassumption} that \eqref{eq:weak-star.convergence} is satisfied for any polynomial on $\mathbb{S}^d$ and hence for all $f \in C( \mathbb{S}^d )$, since the polynomials are dense in $ C( \mathbb{S}^d )$. 

\begin{prop} \label{prop:asymptotic.uniformity}
Given $s > d/2$, a sequence of QMC designs for $\mathbb{H}^s( \mathbb{S}^d )$ is asymptotically uniformly distributed on $\mathbb{S}^d$.
\end{prop}

Asymptotic uniformity of point sets is a relatively weak property:
Even if we restrict our attention to $f \in \mathbb{H}^s( \mathbb{S}^d )$ for all $s$, $f$ non-constant, it is possible to construct asymptotically uniformly distributed sequences of node sets so that the convergence of the quadrature error is as slow as one likes. (Without loss of generality, we can restrict attention to smooth functions $f$ whose integral over $\mathbb{S}^d$ is zero.  One can then assign too many points to regions where $f$ is positive and too few where it is negative, correcting the imbalance as $N \to \infty$ as slowly as one wishes.)

Note that asymptotic uniformity of a sequence of point sets does not imply that the point sets have Property R (as one can always add order $o(N)$ arbitrary points to an $N$-point configuration without changing the limit in \eqref{eq:weak-star.convergence}), nor need a sequence of point sets with Property R be asymptotically uniformly distributed (as one can always delete all points inside a fixed spherical cap without any loss to Property R).

\subsection*{Low-discrepancy sequences}

Unlike the situation for the unit cube, on the sphere $\mathbb{S}^d$ there is no single Koksma-Hlawka inequality (cf., e.g., \cite{DrTi1997}), as evidenced by the many competing Koksma-Hlawka like inequalities proposed in the literature (e.g. \cite{CuFr1997}, \cite{BrDi2012a}). The notion of worst-case error (Definition~\ref{def:wce}) provides a way to bound the error of numerical integration for sufficiently smooth functions $f$, see \eqref{eq:Koksma.like}.
% \begin{equation*}
% \left| \frac{1}{N} \sum_{j=1}^N f( \PT{x}_j ) - \int_{\mathbb{S}^d} f( \PT{x} ) \, \dd \sigma_d( \PT{x} ) \right| \leq \wce( \numint[X_N]; \mathbb{H}^s( \mathbb{S}^d ) ) \, \left\| f \right\|_{\mathbb{H}^s}, \qquad \text{$f \in \mathbb{H}^s( \mathbb{S}^d )$, $s > d/2$.}
% \end{equation*}

The Sobolev space $\mathbb{H}^{(d + 1) / 2}( \mathbb{S}^d )$ with \eqref{eq:K.C.s} as reproducing kernel is special in the sense that the worst-case error $\wce(\numint[X_N]; \mathbb{H}^{(d + 1) / 2}( \mathbb{S}^d ) )$ of a QMC rule with node set $X_N = \{ \PT{x}_1, \dots, \PT{x}_N \}$ has an interpretation as the {\em spherical cap $\mathbb{L}_2$-discrepancy} of $X_N$, defined by
\begin{equation} \label{eq:IL2-discrepancy}
D_{\IL_{2}}^{C}(X_{N}) \DEF \left\{\int_0^\pi \int_{\Sp^{d}} \left|\frac{|X_{N} \cap \mathcal{C}({\PT x}; \theta)|}{N} - \sigma_d( \mathcal{C}({\PT x}; \theta)) \right|^{2} \dd \sigma_d({\PT x}) \, \sin\theta \dd \theta \right\}^{1/2},
\end{equation}
where $\mathcal{C}({\PT x}; \theta)$ denotes a spherical cap as defined in \eqref{eq:spherical.cap}.
Indeed, \emph{Stolarsky's invariance principle} \cite{St1973} asserts that 
\begin{equation} \label{eq:discrepancy.identies}
\left( V_{-1}(\mathbb{S}^d) - \frac{1}{N^2} \sum_{j=1}^N \sum_{i=1}^N \left| \PT{x}_j - \PT{x}_i \right| \right)^{1/2} = \frac{1}{\sqrt{\gamma_d}} D_{\IL_{2}}^{C}(X_N),
\end{equation}
where $V_{-1}( \mathbb{S}^d )$ is given in \eqref{eq:gd.kernel.coeffs} and $\gamma_d$ is given in \eqref{eq:omega.d.ratio}. Recalling equation \eqref{eq:dist.kernel.wce.formula}, we recognize that the left-hand side above is the worst-case error for $\mathbb{H}^{(d+1)/2}( \mathbb{S}^d )$ with kernel $K_{\mathrm{dist}}( \PT{x} \cdot \PT{y} )$ (cf. \eqref{eq:K.C.s}).\footnote{This connection was first noticed in \cite{BrWo2010Manuscript} and rigorously proved in \cite{BrDi2012binpress}.}
As a consequence of \eqref{eq:discrepancy.identies} and \eqref{eq:dist.kernel.wce.formula} %and Theorem~\ref{thm:firstmax} 
for $s = (d+1)/2$, we obtain the following corollary to Theorem \ref{thm:firstmax}.
\begin{cor}
% Let $d \geq 2$. 
Minimizers of the spherical cap $\IL_2$-discrepancy form a sequence of QMC designs for $\mathbb{H}^{(d+1)/2}( \mathbb{S}^d )$.
\end{cor}

A related concept is that of the {\em spherical cap $\IL_{\infty}$-discrepancy} of an $N$-point set $X_N$ on $\mathbb{S}^d$, defined by
\begin{equation} \label{eq:spherical.cap.discrepancy}
D_{\IL_{\infty}}^{C}(X_N) \DEF \sup\Bigg\{ \left|\frac{| X_N \cap \mathcal{C}|}{N} - \sigma_d( \mathcal{C} ) \right| : \text{$\mathcal{C}$ spherical cap in $\mathbb{S}^d$} \Bigg\}.
\end{equation}
Note that $D_{\IL_{2}}^{C}(X_N) \leq \sqrt{2} \, D_{\IL_{\infty}}^{C}(X_N)$. We shall establish in Section~\ref{sec:proofs} the following. 
%
% \begin{prop} \label{prop:Kosma.Hlawka.bound}
% % Let $d \geq 2$ and $s \geq (d + 1) / 2$. 
% \MARKED{Purple}{Let $s \geq (d + 1) / 2$.}
% Then for any $N$-point $X_N$ on $\mathbb{S}^d$ 
% \begin{equation} \label{eq:Koksma.Hlawka.inequality}
% \wce( \numint[X_N]; \mathbb{H}^{s}( \mathbb{S}^d ) ) \leq c_{s,d} \, D_{\IL_{\infty}}^{C}( X_N ),
% \end{equation}
% where $c_{s,d}$ is a positive constant that depends only on the \MARKED{Green}{$\mathbb{H}^{s}( \mathbb{S}^d )$-norm}.
% \end{prop}
\begin{prop} \label{prop:Kosma.Hlawka.bound}
Under the assumptions of Proposition~\ref{prop:wce}, given $s \geq (d + 1) / 2$, every $N$-point configuration on $\mathbb{S}^d$ satisfies
\begin{equation} \label{eq:Koksma.Hlawka.inequality}
\wce( \numint[X_N]; \mathbb{H}^{s}( \mathbb{S}^d ) ) \leq c_{s,d} \, D_{\IL_{\infty}}^{C}( X_N ),
\end{equation}
where $c_{s,d} > 0$ depends on the $\mathbb{H}^{s}( \mathbb{S}^d )$-norm, but is independent of $N$.
\end{prop}

Thus, node sets with small spherical cap discrepancy are of some interest with regard to numerical integration. We remark that a sequence $(X_N)$ of $N$-point configurations on $\mathbb{S}^d$ is asymptotically uniformly distributed if and only if $D_{\IL_{\infty}}^{C}( X_N ) \to 0$ as $N \to \infty$ (see, e.g., \cite{DrTi1997}).

Beck~\cite{Be1984} proved that there is a positive number $c_1$ such that for {\bf any} $N$-point set $Z_N$ on $\mathbb{S}^d$ there exists a spherical cap $\mathcal{C}_N \subset \mathbb{S}^d$ such that
\begin{equation} \label{eq:Beck.lower.bound}
\frac{c_1}{N^{[(d+1)/2]/d}} < \left| \frac{| Z_N \cap \mathcal{C}_N|}{N} - \sigma_d( \mathcal{C}_N ) \right|
\end{equation}
and, by employing a probabilistic argument, that there exist $c_2>0$ and $N$-point sets $Z_N^*$ on $\mathbb{S}^d$ such that
\begin{equation*}
D_{\IL_{\infty}}^{C}( Z_N^* ) < c_2 \, \frac{\sqrt{\log N}}{N^{[(d+1)/2]/d}} \, .
\end{equation*}
This motivates the following definition.

\begin{defn} \label{def:low.discrepancy.sequences}
A sequence $(Z_N)$ of $N$-point configurations on $\mathbb{S}^d$ is said to be a {\em low-discrepancy sequence}\footnote{This terminology is short for \emph{low spherical cap $\IL_\infty$-discrepancy sequence}.} if there exists a positive number $\beta_d$, independent of $N$, such that for all $Z_N$
\begin{equation} \label{eq:discr.inequality}
D_{\IL_{\infty}}^{C}(Z_N) \leq \beta_d \, \frac{\sqrt{\log N}}{N^{[(d+1)/2]/d}}.
\end{equation}
\end{defn}

Let $(Z_N)$ be a low-discrepancy sequence on $\mathbb{S}^d$. Using \eqref{eq:Koksma.Hlawka.inequality} and Theorem~\ref{thm:wce.lower.bound}, we see that for each $s \geq (d + 1)/2$ 
\begin{equation} \label{eq:low.A}
\frac{\beta_1(s,d)}{N^{s/d}} \leq \wce( Q[Z_N]; \mathbb{H}^s( \mathbb{S}^d) ) \leq c_{s,d} \, D_{\IL_{\infty}}^{C}( Z_N ) \leq \beta_2(s,d) \frac{\sqrt{\log N}}{N^{[(d+1)/2]/d}}.
\end{equation}
On the other hand, if $\wce( Q[Z_N]; \mathbb{H}^s( \mathbb{S}^d) ) < 1$, we have by Lemma~\ref{lem:aux.lem} of Section~\ref{sec:proofs} (with $s$ replaced by $(d+1)/2$ and $s^\prime$ replaced by $s$) that, for each $s < ( d + 1 ) / 2$, 
\begin{equation} \label{eq:low.B}
\frac{\beta_1(s,d)}{N^{s/d}} \leq \wce( Q[Z_N]; \mathbb{H}^s( \mathbb{S}^d) ) \leq \beta_3(s,d) \, \frac{( \log N )^{s/(d+1)}}{N^{s/d}}.
\end{equation}
From \eqref{eq:low.B} we see that for every $s \in (d/2, (d+1)/2)$, low-discrepancy sequences on $\mathbb{S}^d$ have almost optimal order of the worst-case error for $\mathbb{H}^s( \mathbb{S}^d) )$, except for a power of $\log N$.
%
% It is unclear whether a low-discrepancy sequence on $\mathbb{S}^d$ \MARKED{cyan}{does or does not} achieve optimal order of the worst-case error if $s > (d + 1)/2$. 
%
This leads to the following natural question: \emph{Do low-discrepancy sequences on $\mathbb{S}^d$ form sequences of QMC designs for $\mathbb{H}^s( \mathbb{S}^d )$ when $s \in (d/2, (d+1)/2)$?}
We remark that to the authors' knowledge no explicit constructions of low-discrepancy configurations on $\mathbb{S}^d$ are known, in contrast to the situation for the unit cube.
\footnote{In \cite{AiBrDi2012arXiv} it is proved that the spherical cap discrepancy of so-called spherical digital nets and spherical Fibonacci points is bounded by $C N^{-1/2}$ (for some explicit $C$), the same rate as for random points.}

\section{QMC designs are better than average}
\label{sec:better.than.average}

As we show in this section, Theorem~\ref{thm:better} is a consequence of a more general result dealing with the average value of the $N$-point energy defined by a positive definite kernel on the sphere.
%
% For $d \geq 2$ and a given 
Given a sequence $( a_\ell )_{\ell \ge 1}$ with $a_\ell \geq 0$ and $\sum_{\ell = 1}^\infty a_\ell Z(d, \ell)$ convergent, we set
\begin{equation} \label{eq:A.pts}
\mathcal{A}( \PT{x} \cdot \PT{y} ) \DEF \mathcal{A}( \PT{x}, \PT{y} ) = \sum_{\ell = 1}^\infty a_\ell \sum_{k=1}^{Z(d,\ell)} Y_{\ell,k}( \PT{x} )  Y_{\ell,k}( \PT{y} ) = \sum_{\ell = 1}^\infty a_\ell Z(d, \ell) P_\ell^{(d)}( \PT{x} \cdot \PT{y} ), \quad \PT{x}, \PT{y} \in \mathbb{S}^d.
\end{equation}
Then $\mathcal{A}$ is a positive definite kernel on $\mathbb{S}^d$ in the sense of Schoenberg~\cite{Sch1942}; indeed, from the addition theorem, for all $N$-point configurations $X_N = \{ \PT{x}_1, \dots, \PT{x}_N \} \subset \mathbb{S}^d$ and all real numbers $\alpha_1, \dots, \alpha_N$,
\begin{equation*}
\sum_{j=1}^N \sum_{i=1}^N \alpha_j \, \mathcal{A}( \PT{x}_j \cdot \PT{x}_i ) \, \alpha_i = \sum_{\ell = 1}^\infty a_\ell \sum_{k=1}^{Z(d,\ell)} \left| \sum_{j=1}^N \alpha_j Y_{\ell,k}( \PT{x}_j ) \right|^2 \geq 0.
\end{equation*}
Consequently, the normalized $N$-point energy of this kernel, defined by 
\begin{equation} \label{eq:A.energy}
\mathcal{A}[ X_N ] \DEF \frac{1}{N^2} \sum_{j=1}^N \sum_{i=1}^N \mathcal{A}( \PT{x}_j \cdot \PT{x}_i ), \qquad X_N = \{ \PT{x}_1, \dots, \PT{x}_N \} \subset \mathbb{S}^d,
\end{equation}
satisfies $\mathcal{A}[ X_N ] \geq 0$.

We are interested in the expected value of the energy \eqref{eq:A.energy} for $N$ random points that are chosen independently and identically distributed with respect to the uniform measure on $\mathbb{S}^d$; namely %given by:
\begin{equation*}
\mathbb{E} \mathcal{A}[X_N] \DEF \int_{\mathbb{S}^d} \cdots \int_{\mathbb{S}^d} \mathcal{A}[ X_N ] \dd \sigma_d( \PT{x}_1 ) \cdots \dd \sigma_d( \PT{x}_N ),
\end{equation*}
which is also known as the {\em spherical average} of $\mathcal{A}[X_N]$ over all $N$-point sets $X_N \subset \mathbb{S}^d$. A straightforward computation yields the following result (see Section~\ref{sec:proofs}).

% \begin{thm} \label{thm:average} 
% Let $d \geq 2$. Then
% \begin{equation} \label{eq:spherical.average}
% \mathbb{E} \mathcal{A}[X_N] = \frac{\mathcal{A}(1)}{N}.
% \end{equation}
% \end{thm}
\begin{thm} \label{thm:average} 
Given the kernel $\mathcal{A}$ as in \eqref{eq:A.pts}, 
\begin{equation} \label{eq:spherical.average}
\mathbb{E} \mathcal{A}[X_N] = \frac{\mathcal{A}(1)}{N}.
\end{equation}
\end{thm}

\begin{rmk}
Theorem~\ref{thm:average} generalizes the result \cite[Theorem~6]{SlWo2009}.
\end{rmk}

We now apply this result to the positive definite kernel $\mathcal{K}^{(s)}$ associated with the reproducing kernel $K^{(s)}$ as given in \eqref{eq:K} for $\mathbb{H}^s( \mathbb{S}^d )$ with $s > d/2$. By Proposition~\ref{prop:wce} 
\begin{equation*}
\mathcal{K}^{(s)}[X_N] = \frac{1}{N^2} \sum_{j=1}^N \sum_{i=1}^N \mathcal{K}^{(s)}({\PT x}_j \cdot {\PT x}_i) = \left[ \wce( \numint[X_N]; \mathbb{H}^s(\mathbb{S}^d) ) \right]^2  ,
\end{equation*}
and hence by Theorem~\ref{thm:average} we obtain that the expected value of the squared worst-case error is given by $\mathcal{K}^{(s)}(1) / N$, from which  Theorem~\ref{thm:better} follows. Consequently, sequences of randomly chosen points on the sphere do not generate QMC designs for $s > d/2$.

By the same token, Theorem~\ref{thm:average} applied to $\mathcal{A} = \mathcal{K}_t^{(s)}$ as given in \eqref{eq:K.t} yields the following.
\begin{prop}
% Let $s > d/2 \geq 1$. 
Given $s > d/2$, for any positive integers $t$ and $N$ (not necessarily depending on each other)
\begin{equation*}
\mathbb{E}\Bigg[ \frac{1}{N^2} \sum_{j=1}^N \sum_{i=1}^N \mathcal{K}_t^{(s)}( \PT{x}_j \cdot \PT{x}_i ) \Bigg] = \frac{\mathcal{K}_t^{(s)}(1)}{N}. 
\end{equation*}
Furthermore, $\mathcal{K}_t^{(s)}(1) \nearrow \mathcal{K}^{(s)}(1)$ as $t \to \infty$. 
\end{prop}
The implication of the first part of this result is that for a sequence of random $N$-point configurations on $\mathbb{S}^d$, the rate of convergence of the square-bracketed expression above is of order $N^{-1}$ regardless of the choice of $t$. We remark that, in contrast, for spherical $t$-designs the square-bracketed expression above vanishes.

If, instead of choosing points randomly over the whole sphere, we stratify our approach by requiring that the $N$ points be randomly chosen from $N$ different equal area subsets of $\mathbb{S}^d$ having small diameter, then on average we will obtain a sequence of QMC designs for $\mathbb{H}^s( \mathbb{S}^d )$ whenever $s \in (d/2, d/2+1)$. In the formal statement of this result we denote by 
\begin{equation*}
\diam A \DEF \sup \big\{ \left| \PT{x} - \PT{y} \right| : \PT{x}, \PT{y} \in A \big\}
\end{equation*}
the diameter of the set $A$.

\begin{thm} \label{thm:equal.area}
Let $(\mathcal{D}_N)$ be a sequence of partitions of $\mathbb{S}^d$ into $N$ equal area subsets $D_{j,N}$, $j = 1, \dots, N$, such that $\diam D_{j,N} \leq c / N^{1/d}$, where $c$ is a positive constant independent of $j$ and $N$. Let $X_N = \{ \PT{x}_{1,N}, \dots, \PT{x}_{N,N} \}$, where $\PT{x}_{j,N}$ is chosen randomly from $D_{j,N}$ with respect to uniform measure on $D_{j,N}$. Then, for $d/2 < s < d/2 + 1$,
\begin{equation} \label{eq:E.wce.equal.area}
\frac{\beta^\prime}{N^{s/d}} \leq \sqrt{\mathbb{E} \big[ \left\{ \wce( \numint[X_N]; \mathbb{H}^s( \mathbb{S}^d ) ) \right\}^2 \big]} \leq \frac{\beta}{N^{s/d}},
\end{equation}
where $\beta^\prime > 0$ and $\beta > 0$ depend on the $\mathbb{H}^s( \mathbb{S}^d )$-norm, but are independent of $N$, and $\beta > 0$ also depends on $(\mathcal{D}_N)$.
% where $\beta$ is a positive constant that depends on the  and $(\mathcal{D}_N)$, but is independent of $N$ and $\beta^\prime$ is a positive constant that only depends on $s$ and $d$.
\end{thm}

We remark that such a sequence of partitions always exists; see \cite{BeCh2008, BoLi1998, Le2009, RaSaZh1994}.

As the next result shows, the stratification strategy does not lead however, on average, to QMC designs for $\mathbb{H}^s( \mathbb{S}^d )$ with $s > d/2 + 1$.

\begin{thm} \label{thm:equal.area.lower.bound}
Let $(\mathcal{D}_N)$ and $X_N$ be as in Theorem~\ref{thm:equal.area}.
Then for $s > d/2 + 1$, with $2 s - d$ not an even integer,\footnote{A similar result for $2s - d$ an even integer can be obtained using an appropriate kernel as discussed in Section~\ref{sec:generalized.distance.kernel}.}
\begin{equation*}
\sqrt{ \mathbb{E} \big[ \left\{ \wce( \numint[X_N]; \mathbb{H}^s( \mathbb{S}^d ) ) \right\}^2 \big] } \geq \frac{\beta}{N^{(d/2+1)/d}},
\end{equation*}
where $\beta > 0$ depends on the $\mathbb{H}^s( \mathbb{S}^d )$-norm
% $d$, $s$ 
and $(\mathcal{D}_N)$, but is independent of $N$.
\end{thm}

Other concepts of randomness can be also considered. Armentano, Beltr{\'a}n, and Shub~\cite{ArBeSh2011} study point configurations on $\mathbb{S}^2$ that are derived from the zeros of random polynomials. This will be a topic of future research.

\section{Numerical Experiments}
\label{sec:numerics}

\subsection{Point sets}

Many different sequences of point sets are potential candidates to be QMC design sequences
for $\mathbb{H}^s( \mathbb{S}^2 )$ and some $s > 1$.
We consider the following point sets $X_N$.
%A large variety of point sets for the sphere $\mathbb{S}^2$ have been proposed
%which are putative QMC designs for different spaces $\mathbb{H}^s( \mathbb{S}^2 )$, $s > 1$.
%We consider the following point sets $X_N$, all (except the spherical designs)
%having $N = (t+1)^2$ points for $t = 1,2,\ldots,100$:
\begin{itemize}
\item Pseudo-random points, uniformly distributed on the sphere.
\item Equal area points based on an algorithm given in \cite{RaSaZh1994}.
\item Fekete points which maximize the determinant for polynomial interpolation \cite{SlWo2004}.
\item Coulomb energy points, which minimize
\begin{equation*}
\sum_{j=1}^N \sum_{i=1}^N \frac{1}{\left| \PT{x}_j - \PT{x}_i \right|}.
\end{equation*}
\item Log energy points, which minimize
\begin{equation*}
\sum_{j=1}^N \sum_{i=1}^N \log \frac{1}{\left| \PT{x}_j - \PT{x}_i \right|} .
\end{equation*}
\item Generalized spiral points (cf. \cite{RaSaZh1994,Ba2000}),
with spherical coordinates $(\theta_j, \phi_j)$ given by
\begin{equation*}
z_j = 1 - \frac{2j-1}{N}, \quad \theta_j = \cos^{-1}(z_j), \quad \phi_j = 1.8 \sqrt{N}\theta_j \mod 2\pi,
 \quad j = 1,\ldots,N.
\end{equation*}
\item Distance points, which maximize 
\begin{equation*}
\sum_{j=1}^N \sum_{i=1}^N \left| \PT{x}_j - \PT{x}_i \right|.
\end{equation*}
\item Spherical $t$-designs with $N = \lceil (t+1)^2/2 \rceil + 1$ points.
\end{itemize}
All the point sets that are characterized by optimizing a criterion are faced with the
difficulty of many local optima.
Thus, for larger values of $N$, these point sets have objective values near,
but not necessarily equal to, the global optimum.

QMC designs for  $\mathbb{H}^s\left( \mathbb{S}^2 \right)$, for every $s > 1$ ($s$ not an integer),
could be calculated by optimizing the expression \eqref{eq:wce.4.K.dg.kernel} in terms of the generalized distance
$|\PT{x}_j - \PT{x}_i|^{2s-2}$
(including the low order polynomial $\mathcal{Q}_L$ or by imposing the additional constraints
that the point set is an $L$-design), as discussed in Section~\ref{sec:generalized.distance.kernel}.
We restrict attention to the point sets listed above, which are available from the website \cite{Wo2012www}.
Here the criterion \eqref{eq:wce.4.K.dg.kernel} is used to generate points only for the case $s = 3/2$
(maximizing distance sums).
In all cases, except the last, the number of points $N$ was taken to be a perfect square.

\subsection{Worst case error}

\begin{figure}[h!t]
\begin{center}
\includegraphics[scale=0.9]{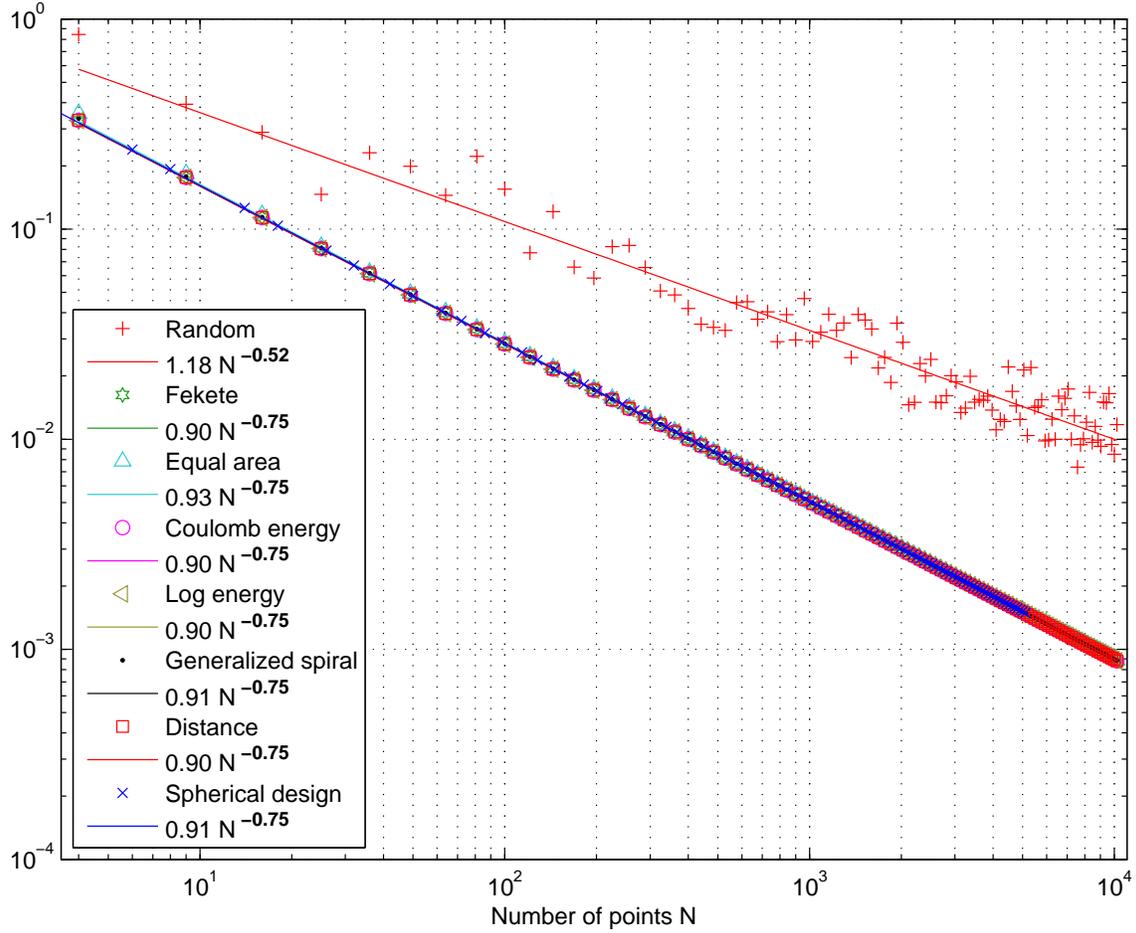}
\caption{Worst case error for $\mathbb{H}^s( \mathbb{S}^2 )$ and $s = 3/2$}
\label{F:wce15}
\end{center}
\end{figure}

The worst-case errors in $\mathbb{H}^s \left( \mathbb{S}^2\right)$ for $s = 3/2$
are illustrated in Figure~\ref{F:wce15}.
For all point sets, the worst case error with $s = 3/2$ is calculated
using \eqref{eq:dist.kernel.wce.formula} and the distance kernel.
Apart from the random points (which are not QMC designs, see Theorem~\ref{thm:better})
all the point sets have a worst-case error for $s = 3/2$ that seems to decay like $N^{-3/4}$, 
implying that they are all QMC designs for $s = 3/2$.

Spherical designs with $N = \mathcal{O}(t^2)$ are QMC designs for
$\mathbb{H}^s\left( \mathbb{S}^2 \right)$ and all $s > 1$.
From Section~\ref{sec:generalized.distance.kernel} and Theorem~\ref{thm:firstmax},
the distance points are provably QMC designs for $s = 1.5$.
The rate of decay for equal area points fits well with Theorem~\ref{thm:equal.area},
which established that randomized equal area points are also QMC designs for $1 < s < 2$.
Moreover, from Theorem~\ref{thm:equal.area.lower.bound},
randomized equal area points cannot do better than this.
Other than the distance points,
it has yet to be established rigorously that the non-random point sets
are QMC designs for $s = 1.5$.
It is rather curious that, for the non-random sequences,
the computed worst case errors in Figure~\ref{F:wce15} essentially lie on the same curve.

\begin{figure}[h!t]
\begin{center}
\includegraphics[scale=0.9]{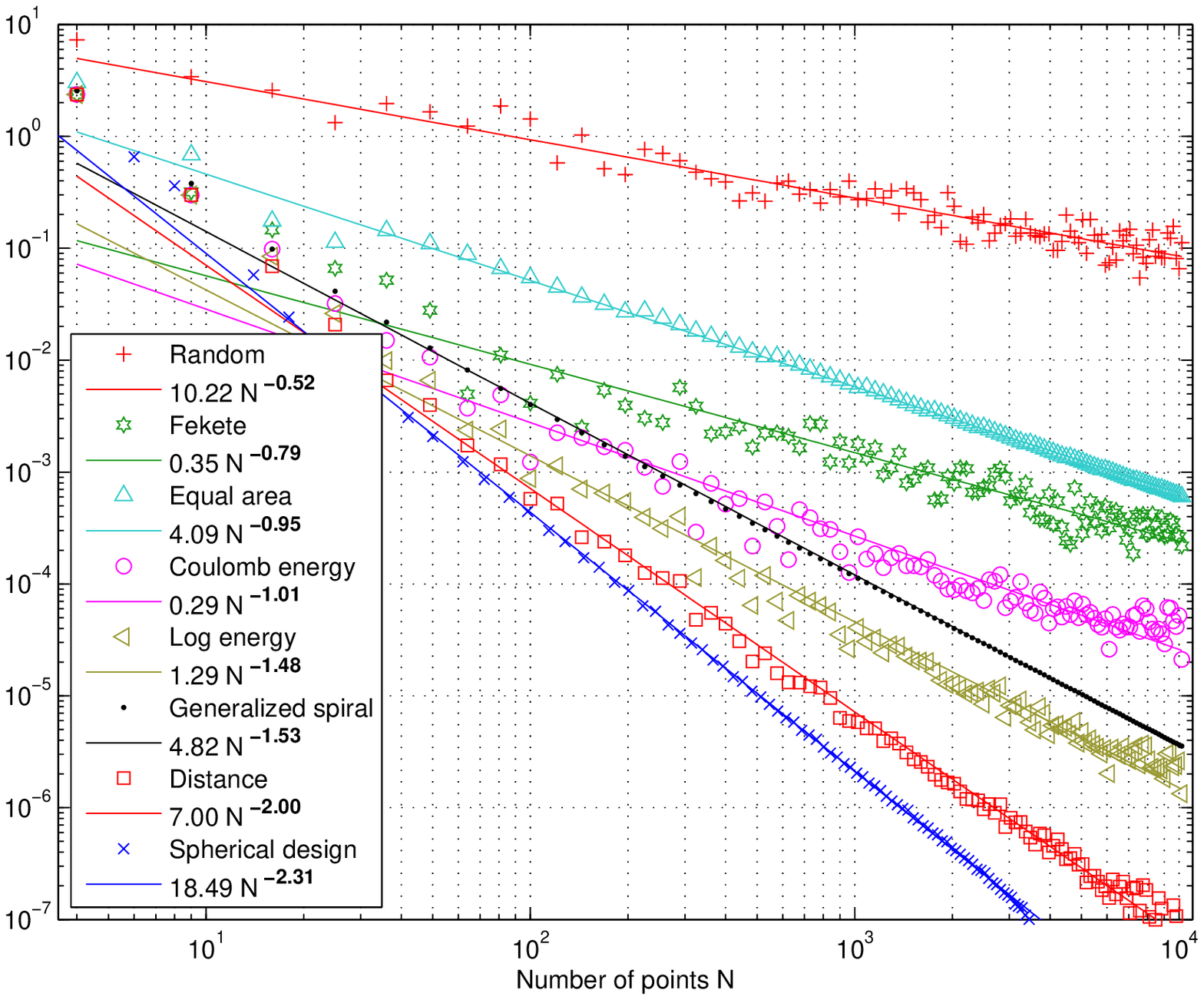}
\caption{Worst case error for $\mathbb{H}^s( \mathbb{S}^2 )$ and $s = 4.5$}
\label{F:wce45}
\end{center}
\end{figure}

Figure~\ref{F:wce45} plots the worst-case errors for $s = 4.5$ and estimates the
rate of decay by finding a least squares fit of the form $\alpha N^{-\beta}$
for $10 \leq N \leq 10^4$ (except for the spherical designs which use $10\leq N \leq 5 \times 10^3$).
For all point sequences the worst case error with $s = 4.5$ is calculated
using the generalized distance kernel and \eqref{eq:wce.4.K.dg.kernel} (for which $L = 3$).
As expected from Theorem~\ref{thm:better}, for random points the worst case error still decays like $N^{-1/2}$.

All the point sets, except for the spherical designs,
exhibit varying rates of decay slower than $\mathcal{O}( N^{-s/2} ) = \mathcal{O}( N^{-2.25} )$,
indicating that their effectiveness for equal weight numerical integration
on $\mathbb{H}^s( \mathbb{S}^2 )$ when $s \geq 4.5$
is less than optimal; cf. Theorem~\ref{thm:best.rate}.

\subsection{Integrating a smooth function}
One expects from the error bound \eqref{eq:Koksma.like}
that a putative sequence $(X_N)$ of QMC designs for
$\mathbb{H}^s( \mathbb{S}^2 )$, $s > 1$,
will play out its full strength when tested with a smooth test function.
Some caution is needed in the choice of this function
in order to avoid having an integrand that is accidentally too easy.
Our choice is the {\em Franke} function for the sphere \cite{Re1988} defined by
\begin{align*}
f\left( x,y,z \right)
  &\DEF  0.75\exp (-(9x-2)^{2}/4-(9y-2)^{2}/4-(9z-2)^{2}/4) \\
  &\phantom{=} +0.75\exp (-(9x+1)^{2}/49-(9y+1)/10-(9z+1)/10) \\
  &\phantom{=} +0.5\exp (-(9x-7)^{2}/4-(9y-3)^{2}/4-(9z-5)^{2}/4) \\
  &\phantom{=} -0.2\exp (-(9x-4)^{2}-(9y-7)^{2}-(9z-5)^{2}), \qquad (x,y,z)^T \in\mathbb{S}^2,
\end{align*}
which is in $C^{\infty}(\mathbb{S}^{2})$ and for which
\begin{equation*}
\int_{\mathbb{S}^2} f(\PT{x}) \mathrm{d} \sigma_2( \PT{x} ) =  0.5328652500843890\ldots .
%\int_{\mathbb{S}^2} f(\PT{x}) \mathrm{d} \sigma_2( \PT{x} ) = 0.532865250084389017242406502772
\end{equation*}
As $f \in \mathbb{H}^s(\mathbb{S}^2)$ for all $s > 1$,
the integration error for a particular sequence $(X_N)$ of QMC designs
with $s^*$ given by \eqref{eq:s.star},
must decay at least as fast as $O\left(N^{-s^*/2 + \epsilon}\right)$
for any $\epsilon > 0$.
\begin{figure}[h!t]
\begin{center}
  \includegraphics[scale=0.9]{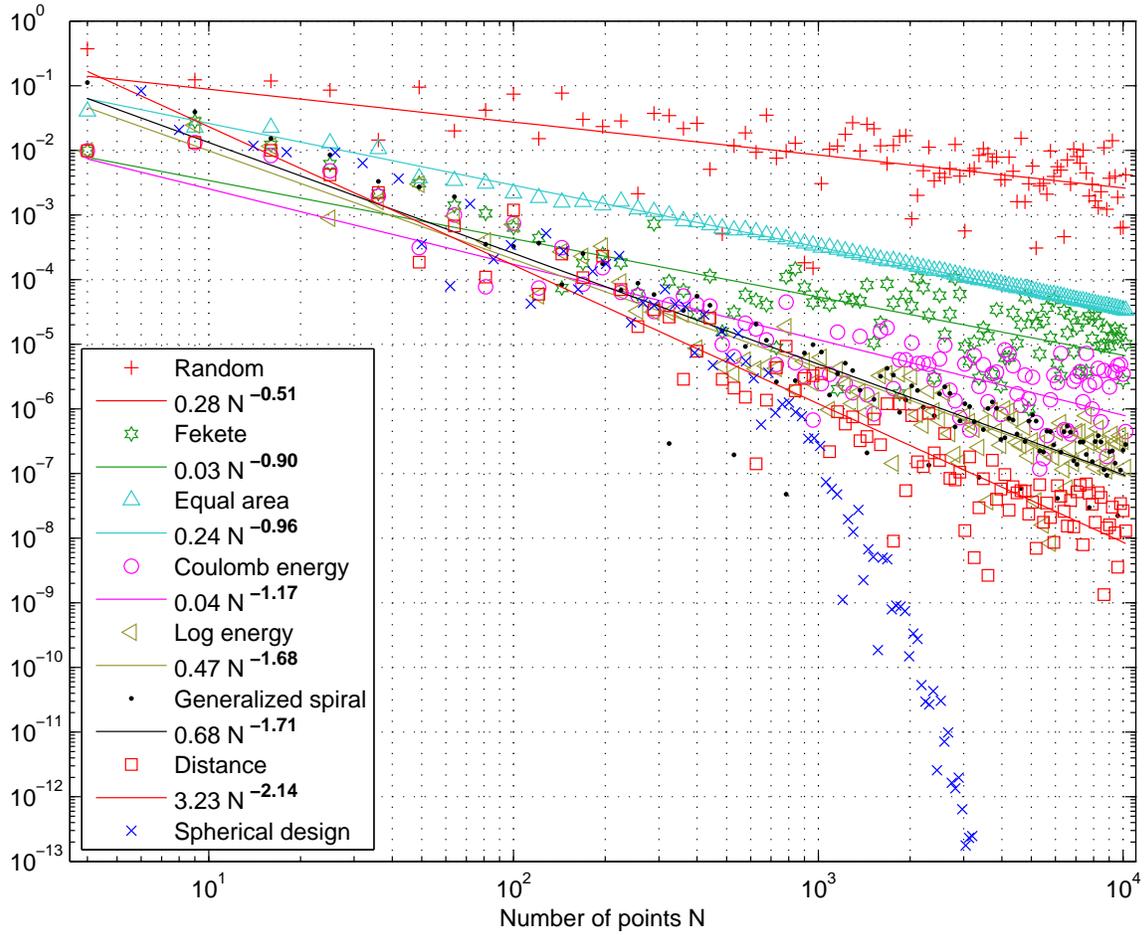}
\caption{Numerical integration errors for the Franke function}
\label{F:Ferr}
\end{center}
\end{figure}
In Figure~\ref{F:Ferr}, the faster than algebraic decay of the numerical integration error
for spherical designs is apparent.

\begin{table}[h!t]
\begin{center}
\caption{\label{T:sstar} Estimates of $s^*$ for $d = 2$}
\begin{tabular}{l|l}
\textbf{Point set} & $s^*$ \\
\hline
Fekete & $1.5$ \\
Equal area & $2$ \\
Coulomb energy & $2$ \\
Log energy & $3$ \\
Generalized spiral & $3$ \\
Distance & $4$ \\
Spherical designs & $\infty$
\end{tabular}
\end{center}
\end{table}

\subsection{Estimating $s^*$}
In order to estimate the value of $s^*$ defined by \eqref{eq:s.star}
we need to calculate the worst-case error for
$\mathbb{H}^s\left( \mathbb{S}^2 \right)$ with a value of $s > s^*$
(cf. Theorems~\ref{thm:rough} and \ref{thm:best.rate}).
Similarly, if the error for the Franke function is decaying approximately like $N^{-{\bar s}/2}$,
then that point set can only be a QMC design for $s \leq \bar s$.

Some conjectured value of $s^*$ are given in Table~\ref{T:sstar},
based on the results in Figure~\ref{F:wce45},
similar experiments with different values of $s$,
and the results in Figure~\ref{F:Ferr}.
For example, the equal area points have a estimated rate of decay
$N^{-0.95}$ for the worst case error with $s = 4.5$ in Figure~\ref{F:wce45},
while the error decays like $N^{-0.96}$ for the Franke function in Figure~\ref{F:Ferr},
leading us to conjecture that $s^* \approx 2$.

Determining the precise value of $s^*$ is very much an open question.

% \clearpage

\section{Proofs}
\label{sec:proofs}

Throughout this proof section we use the shortened notation $\numint_N$ for a QMC rule $\numint[X_N]$ defined by a node set $X_N = \{ \PT{x}_1, \dots, \PT{x}_N \} \subset \mathbb{S}^d$. 

\subsection{Proofs of Section~\ref{sec:intro} results}

The proof of Theorem~\ref{thm:rough} requires the following lemma.
\begin{lem} \label{lem:aux.lem}
% Let $s > d/2 \geq 1$. If $\wce( \numint_N; \mathbb{H}^s( \mathbb{S}^d ) ) < 1$, then for $d/2 < s^\prime < s$,
Given $s > d/2$, if $\wce( \numint_N; \mathbb{H}^s( \mathbb{S}^d ) ) < 1$, then
\begin{equation} \label{eq:wce.inequality.B}
\wce( \numint_N; \mathbb{H}^{s^\prime}( \mathbb{S}^d ) ) < c(d,s,s^\prime) \left[ \wce( \numint_N; \mathbb{H}^s( \mathbb{S}^d ) ) \right]^{s^\prime / s}, \qquad d/2 < s^\prime < s,
\end{equation}
where $c(d,s,s^\prime) > 0$ depends on the norms for $\mathbb{H}^s( \mathbb{S}^d )$ and $\mathbb{H}^{s^\prime}( \mathbb{S}^d )$, but is independent of~$N$.
\end{lem}

\begin{proof}
Relation \eqref{eq:wce.inequality.B} follows from Theorem~3.1 in Brandolini et al.~\cite{BrChCoarXiv1012.5409v1}. For the sake of completeness we give here a proof  of \eqref{eq:wce.inequality.B} along the lines of the proof of Theorem~3.1 tailored to our needs and specifically to the case of spheres $\mathbb{S}^d$. Throughout this proof we use the canonical kernel $K_\mathrm{can}^{(s)}$. 

Writing $1 / ( 1 + \lambda_\ell )^{s^\prime}$ in terms of a Laplace transform (cf. \cite[Eq.~5.9.1]{DLMF2010.05.07}),
\begin{equation*}
\frac{1}{( 1 + \lambda_\ell )^{s^\prime}} = \frac{1}{\gammafcn(s^\prime)} \int_0^\infty e^{-(1 + \lambda_\ell) \tau} \tau^{s^\prime-1} \dd \tau,
\end{equation*}
and applying Proposition~\ref{prop:wce} we obtain for all $d/2 < s^\prime \leq s$
\begin{equation} \label{eq:wce.int.repr}
\left[ \wce( \numint_{N}; \mathbb{H}^{s^\prime}( \mathbb{S}^d ) ) \right]^2 = \frac{1}{\gammafcn(s^\prime)} \int_0^\infty \tau^{s^\prime-1} e^{- \tau} g( \tau ) \dd \tau.
\end{equation}
Here
\begin{equation*}
g(\tau) \DEF g(\tau; \PT{x}_1, \dots, \PT{x}_N ) \DEF \sum_{\ell=1}^\infty e^{- \lambda_\ell \tau} \sum_{k=1}^{Z(d,\ell)} \left| \frac{1}{N} \sum_{j=1}^N Y_{\ell,k}( \PT{x}_j ) \right|^2 = \frac{1}{N^2} \sum_{j=1}^N \sum_{i=1}^N \mathcal{H}(\tau, \PT{x}_j \cdot \PT{x}_i ),
\end{equation*}
where calligraphic $\mathcal{H}$ denotes the \emph{heat kernel} with the constant term removed:
\begin{equation*}
H(\tau, \PT{x}, \PT{y}) \DEF \sum_{\ell=0}^\infty e^{- \lambda_\ell \tau} \, \sum_{k=1}^{Z(d,\ell)} Y_{\ell,k}( \PT{x} ) Y_{\ell,k}( \PT{y} ) = \sum_{\ell=0}^\infty e^{- \lambda_\ell \tau} \, Z(d,\ell) \, P_\ell^{(d)}( \PT{x} \cdot \PT{y} ), \qquad \PT{x}, \PT{y} \in \mathbb{S}^d,
\end{equation*}
which is the fundamental solution to the heat equation $\partial u / \partial \tau + \Delta_d^* \, u = 0$ on $\mathbb{R}_+ \times \mathbb{S}^d$. 
Interchanging integration and summation in \eqref{eq:wce.int.repr} is justified, because the heat kernel is uniformly continuous on $[\eps, \infty) \times \mathbb{S}^d \times \mathbb{S}^d$ for $\eps > 0$.

Let $d/2 < s^\prime < s$ and set $\eps \DEF [\wce( \numint_{N}; \mathbb{H}^{s}( \mathbb{S}^d ) )]^{2/s}$. (Then $\eps < 1$ by assumption.) We split the right-hand side of \eqref{eq:wce.int.repr} into three parts and use the following estimates: 
(i) For ``large values'' of $\tau$,
\begin{align*}
\frac{1}{\gammafcn( s^\prime )} \int_1^\infty \tau^{s^\prime - 1} e^{-\tau} g(\tau) \dd \tau
&< \frac{1}{\gammafcn( s^\prime )} \int_1^\infty \tau^{s - 1} e^{-\tau} g(\tau) \dd \tau
 \leq \frac{\gammafcn( s )}{\gammafcn( s^\prime )} \, \frac{1}{\gammafcn( s )} \int_0^\infty \tau^{s - 1} e^{-\tau} g(\tau) \dd \tau \\
&= \frac{\gammafcn( s )}{\gammafcn( s^\prime )} \, \eps^{s}
 < \frac{\gammafcn( s )}{\gammafcn( s^\prime )} \, \eps^{ s^\prime},
\end{align*}
where we used \eqref{eq:wce.int.repr} with $s^\prime$ replaced by $s$.

(ii) For ``medium values'' of $\tau$,
\begin{align*}
\frac{1}{\gammafcn( s^\prime )} \int_{\eps/2}^1 \tau^{s^\prime - 1} e^{-\tau} g(\tau) \dd \tau
&= \frac{1}{\gammafcn( s^\prime )} \int_{\eps/2}^1 \tau^{s - 1} \tau^{s^\prime-s} e^{-\tau} g(\tau) \dd \tau
 \leq \frac{(\eps/2)^{s^\prime-s}}{\gammafcn( s^\prime )} \int_{\eps/2}^1 \tau^{s - 1} e^{-\tau} g(\tau) \dd \tau \\
&< (\eps/2)^{s^\prime-s} \frac{\gammafcn( s )}{\gammafcn( s^\prime )} \, \frac{1}{\gammafcn( s )} \int_{0}^\infty \tau^{s - 1} e^{-\tau} g(\tau) \dd \tau
 = 2^{s-s^\prime} \frac{\gammafcn( s )}{\gammafcn( s^\prime )} \eps^{s^\prime-s} \eps^s \\
&= 2^{s-s^\prime} \frac{\gammafcn( s )}{\gammafcn( s^\prime )} \eps^{s^\prime}.
\end{align*}

(iii) For ``small values'' of $\tau$, we appeal to the small time Gaussian estimate on the heat kernel (cf. \cite{Va1990}); that is, if $ 0 < \tau < \eps/2$, then for some $c> 0$
\begin{equation*}
\tau^{d/2} H(\tau, \PT{x}, \PT{y} ) \leq c \, \eps^{d/2} H( \eps, \PT{x}, \PT{y} ),
\end{equation*}
which in turn implies, on the assumption that $g(\tau)$ is uniformly bounded on $[0,1)$,
\begin{equation*}
\tau^{d/2} g(\tau) = \frac{1}{N^2} \sum_{j=1}^N \sum_{i=1}^N \tau^{d/2} H(\tau, \PT{x}_j, \PT{x}_i ) - \tau^{d/2} \leq c \, \eps^{d/2} \left( g(\eps) + 1 \right) \leq c^{\prime\prime\prime} \eps^{d/2} .
\end{equation*}
Postponing the proof of the uniform boundedness, it then follows that
\begin{align*}
\frac{1}{\gammafcn( s^\prime )} \int_0^{\eps/2} \tau^{s^\prime - 1} e^{-\tau} g(\tau) \dd \tau
&= \frac{1}{\gammafcn( s^\prime )} \int_0^{\eps/2} \tau^{s^\prime - d/2 - 1} e^{-\tau} \, t^{d/2} g(\tau) \dd \tau \\
&\leq \frac{c^{\prime\prime\prime} }{\gammafcn( s^\prime )} \eps^{d/2} \int_0^{\eps/2} \tau^{s^\prime - d/2 - 1} e^{-\tau} \dd \tau \\
&\leq \frac{c^{\prime\prime\prime} }{\gammafcn( s^\prime )} \eps^{d/2} \int_0^{\eps/2} \tau^{s^\prime - d/2 - 1} \dd \tau
 = \frac{c^{\prime\prime\prime} }{\gammafcn( s^\prime )} \frac{\eps^{s^\prime}}{s^\prime-d/2}.
\end{align*}

From (i), (ii) and (iii) we get the required estimate
\begin{align*}
\left[ \wce( \numint_{N}; \mathbb{H}^{s^\prime}( \mathbb{S}^d ) ) \right]^2
&\leq \left( \frac{\gammafcn( s )}{\gammafcn( s^\prime )} + 2^{s-s^\prime} \frac{\gammafcn( s )}{\gammafcn( s^\prime )} + \frac{c^{\prime\prime\prime} }{\gammafcn( s^\prime )} \frac{1}{s^\prime-d/2} \right) \eps^{s^\prime} \\
&\leq c^{\prime\prime\prime\prime} \left[ \wce( \numint_{N}; \mathbb{H}^{s}( \mathbb{S}^d ) ) \right]^{2 s^\prime / s}.
\end{align*}

It remains to prove uniform boundedness of the $g(\tau)$ for $0 \leq \tau <1$. (We use \eqref{eq:Koksma.like} and the fact that $\int_{\mathbb{S}^d} \mathcal{H}(\tau, \PT{x}, \PT{x}_k ) \dd \sigma_d( \PT{x} ) = 0$.)
\begin{align*}
0 < g(\tau)
&= \frac{1}{N} \sum_{j=1}^N \frac{1}{N} \sum_{i=1}^N \mathcal{H}(\tau, \PT{x}_j, \PT{x}_i )
 \leq \wce( \numint_{N}; \mathbb{H}^{s}( \mathbb{S}^d ) ) \left\| \frac{1}{N} \sum_{j=1}^N \mathcal{H}(\tau, \PT{\cdot}, \PT{x}_j ) \right\|_{\mathbb{H}^s} \\
&= \wce( \numint_{N}; \mathbb{H}^{s}( \mathbb{S}^d ) ) \left\| \sum_{\ell=1}^\infty \sum_{k=1}^{Z(d,\ell)} e^{-\lambda_\ell \, \tau} \left( \frac{1}{N} \sum_{j=1}^N Y_{\ell,k}( \PT{x}_j ) \right) Y_{\ell,k}( \PT{\cdot} ) \right\|_{\mathbb{H}^s} \\
&= \wce( \numint_{N}; \mathbb{H}^{s}( \mathbb{S}^d ) ) \left[ \sum_{\ell=1}^\infty \sum_{k=1}^{Z(d,\ell)} \left( 1 + \lambda_\ell \right)^s e^{-2\lambda_\ell \tau} \left| \frac{1}{N} \sum_{j=1}^N Y_{\ell,k}( \PT{x}_j ) \right|^2 \right]^{1/2} \\
&\leq \wce( \numint_{N}; \mathbb{H}^{s}( \mathbb{S}^d ) ) \left[ \sum_{\ell=1}^\infty \sum_{k=1}^{Z(d,\ell)} \left( 1 + \lambda_\ell \right)^{-s} \left| \frac{1}{N} \sum_{j=1}^N Y_{\ell,k}( \PT{x}_j ) \right|^2 \right]^{1/2} \sup_{\ell \geq 1} \left\{ \left( 1 + \lambda_\ell \right)^s e^{-\lambda_\ell \tau} \right\} \\
&= \left[ \wce( \numint_{N}; \mathbb{H}^{s}( \mathbb{S}^d ) ) \right]^2 \sup_{\ell \geq 1} \left\{ \left( 1 + \lambda_\ell \right)^s e^{-\lambda_\ell \tau} \right\} = \tau^s \, \sup_{\ell \geq 1} \left\{ \left( 1 + \lambda_\ell \right)^s e^{-\lambda_\ell \tau} \right\}.
\end{align*}
The function $\left( 1 + \lambda \right)^s e^{-\lambda \tau}$ has a unique maximum at $\lambda = s / \tau - 1$ with value $s^s \tau^{-s} e^{\tau-s}$. Hence
\begin{equation*}
0 < g(\tau ) \leq s^s e^{\tau-s} \leq s^s e^{1-s}.
\end{equation*}
\end{proof}

\begin{proof}[Proof of Theorem~\ref{thm:rough}]
If $(X_N)$ is a sequence of $N$-point QMC designs for $\mathbb{H}^s( \mathbb{S}^d )$ for $s > d/2$, then $\wce( \numint_{N}; \mathbb{H}^s( \mathbb{S}^d ) ) \to 0$ as $N \to \infty$. Thus $\wce( \numint_{N}; \mathbb{H}^s( \mathbb{S}^d ) ) < 1$ for all $N > N_0$ for some $N_0>0$. By \eqref{eq:wce.inequality.B}
\begin{equation*}
\wce( \numint_{N}; \mathbb{H}^{s^\prime}( \mathbb{S}^d ) ) < c(d,s,s^\prime) \left[ \wce( \numint_{N}; \mathbb{H}^s( \mathbb{S}^d ) ) \right]^{s^\prime / s} < c(d,s,s^\prime) \left[ \frac{c(s,d)}{N^{s/d}} \right]^{s^\prime / s} = \frac{c^{\prime\prime}}{N^{s^\prime/d}}%  \quad \forall N > N_0
\end{equation*}
for every $N > N_0$ for each $d/2 < s^\prime < s$. The finitely many exceptions with $N \leq N_0$ satisfy the last inequality with a possibly larger constant $c^{\prime\prime}$, depending only on the norms of $\mathbb{H}^s( \mathbb{S}^d )$ and $\mathbb{H}^{s^\prime}( \mathbb{S}^d )$.
Consequently, $(X_N)$ is a sequence of QMC designs for $\mathbb{H}^{s^\prime}( \mathbb{S}^d )$ for $d/2 < s^\prime < s$. 
\end{proof}

\begin{proof}[Proof of Theorem~\ref{thm:derrived.gen.sequ}]
Let $d \geq 2$. By Theorem~\ref{thm:bondarenko.et.al} there exists a sequence $(Y_{N_t})$ of spherical $t$-designs $Y_{N_t}$ with $N_t \DEF m_d \, t^d$ points ($t \geq 1$) for some suitably large positive integer $m_d$ %constant $c_d$. 
Furthermore, the theorem states that there exist spherical $t$-designs for every cardinality $\geq m_d \, t^d$. Thus we can fill the gaps in the sequence $(Y_{N_t})$ by adding spherical $t$-designs with $N$ points for $N_t < N < N_{t+1}$. If necessary we choose for $N = 2, \dots, m_d$ spherical $1$-designs with $N$ points; that is, configurations with centroid $0$. 
This gives a new sequence $( \widehat{Y}_N )_{N \geq 2}$.

By Theorem~\ref{thm:s.d.prop} there exists a constant $C(s,d) > 0$ such that
\begin{equation} \label{eq:inequality.A}
\sup_{\substack{f \in \mathbb{H}^s( \mathbb{S}^d ), \\ \| f \|_{\mathbb{H}^s} \leq 1}} \Bigg| \frac{1}{N} \sum_{\PT{y} \in \widehat{Y}_N} f( \PT{y} ) - \int_{\mathbb{S}^d} f( \PT{y} ) \dd \sigma_d( \PT{y} ) \Bigg| \leq \frac{C(s,d)}{t^s} = \frac{C(s,d)}{N^{s/d}} \left( \frac{N^{1/d}}{t} \right)^s
\end{equation}
for all $N \geq 1$. Since
\begin{equation*}
c_d^{1/d} = \frac{N_t^{1/d}}{t} \leq \frac{N^{1/d}}{t} \leq \frac{N_{t+1}^{1/d}}{t} = c_d^{1/d} \, \frac{t+1}{t} \qquad \text{for all $N_t \leq N \leq N_{t+1}$,}
\end{equation*}
the right-hand side of \eqref{eq:inequality.A} satisfies for all $N \geq 1$
\begin{equation*}
\frac{C(s,d)}{N^{s/d}} \left( \frac{N^{1/d}}{t} \right)^s \leq \frac{c_d^{s/d} C(s,d) \left( 1 + 1/t \right)^s}{N^{s/d}} \leq \frac{C^\prime(s,d)}{N^{s/d}}, \qquad C^\prime(s,d) \DEF 2^s c_d^{s/d} C(s,d).
\end{equation*}
Consequently, $( \widehat{Y}_N )_{N \geq 2}$ is a sequence of generic QMC designs.
\end{proof}

\subsection{Proofs of Section~\ref{sec:wce} results}

\begin{proof}[Proof of Theorem~\ref{thm:minimizers}]
For $N \geq 2$ let $X_N^* = \{ \PT{x}_{1,N}^*, \dots, \PT{x}_{N,N}^* \}$ be as in Theorem~\ref{thm:minimizers} and $Y_N = \{ \PT{y}_{1,N}, \dots, \PT{y}_{N,N} \}$ be as in Theorem~\ref{thm:derrived.gen.sequ}. By minimality of the $X_N^*$'s, for every $N \geq 2$ 
\begin{equation*}
\sum_{j=1}^N \sum_{i=1}^N \, \mathcal{K}^{(s)}({\PT x}_{j,N}^* \cdot {\PT x}_{i,N}^* ) \leq \sum_{j=1}^N \sum_{i=1}^N \, \mathcal{K}^{(s)}({\PT y}_{j,N} \cdot {\PT y}_{i,N} ).
\end{equation*}
Hence, by Proposition~\ref{prop:wce} 
\begin{equation*} 
\wce( \numint[X_N^*]; \mathbb{H}^s( \mathbb{S}^d ) ) \leq \wce( \numint[Y_N]; \mathbb{H}^s( \mathbb{S}^d ) ) \leq \frac{c(s,d)}{N^{s/d}},
\end{equation*}
where the last inequality follows by Theorem~\ref{thm:derrived.gen.sequ}.
\end{proof}

For an $N$-point configuration $X_N$ on $\mathbb{S}^d$ the worst-case error formula \eqref{eq:earlywce} can be expressed in terms of the polynomial part \eqref{eq:K.t} and the tail part
\begin{equation} \label{eq:Fourier.expansion.tail.part}
\mathfrak{T}_t^{(s)}( z ) \DEF \sum_{\ell=t+1}^\infty a_\ell^{(s)} Z(d,\ell) P_\ell^{(d)}( z ), \qquad z \in [-1,1];
% = \sum_{\ell=t+1}^\infty a_\ell^{(s)} \sum_{k=1}^{Z(d,\ell)} Y_{\ell,k}( \PT{x} ) Y_{\ell,k}( \PT{y} ), \qquad \PT{x}, \PT{y} \in \mathbb{S}^d;
\end{equation}
that is,
\begin{equation} \label{eq:wce3}
\left[ \wce( \numint[X_N]; \mathbb{H}^s( \mathbb{S}^d ) ) \right]^2 = \frac{1}{N^2}\sum_{j=1}^{N} \sum_{i=1}^{N} \, \mathcal{K}^{(s)}_t({\PT x}_j \cdot {\PT x}_i) + \frac{1}{N^2} \sum_{j=1}^{N} \sum_{i=1}^{N} \, \mathfrak{T}_t^{(s)}({\PT x}_j \cdot {\PT x}_i).
\end{equation}

Since $\mathcal{K}^{(s)}_t$ and $\mathfrak{T}_t^{(s)}$ are positive definite in the sense of Schoenberg, both double sums in \eqref{eq:wce3} are non-negative (cf. Lemma~\ref{lem:positive.definiteness}). Consequently, a sequence of QMC designs must satisfy the following necessary conditions.
\begin{prop} \label{prop:necessary.conditions}
For a given sequence $(X_N)$ of QMC designs $X_N = \{ \PT{x}_{1,N}, \dots, \PT{x}_{N,N} \}$ for $\mathbb{H}^s( \mathbb{S}^d )$, $s > d/2$, there exists $c(s,d) > 0$ independent of $N$ such that, for all integers $t$ %with $t \asymp N^{1/d}$,
\begin{equation*}
\frac{1}{N^2}\sum_{j=1}^{N} \sum_{i=1}^{N} \, \mathcal{K}^{(s)}_t({\PT x}_{j,N} \cdot {\PT x}_{i,N}) \leq \frac{[ c(s,d) ]^2}{N^{2s/d}}, \qquad \frac{1}{N^2} \sum_{j=1}^{N} \sum_{i=1}^{N} \, \mathfrak{T}_t^{(s)}({\PT x}_{j,N} \cdot {\PT x}_{i,N}) \leq \frac{[ c(s,d) ]^2}{N^{2s/d}}.
\end{equation*}
\end{prop}

As discussed in \cite{BrHe2007}, one cannot expect to get the desired estimate $\mathcal{O}(t^{-2s})$, which translates into $\mathcal{O}(N^{-2s/d})$ if $N \asymp t^d$, simply by term wise estimation of the energy corresponding to the tail portion $\mathfrak{T}_t^{(s)}$ of the reproducing kernel.
The papers \cite{BrHe2007,HeSl2005,HeSl2006} thus devised a method of splitting the reproducing kernel into a polynomial part, which could be integrated exactly by an appropriate numerical integration rule, and a remainder part small enough to be bounded in a simple way. 
In \cite{BrWo2010Manuscript} this method is applied to derive the following estimate for the worst-case error for a sequence $(X_N)$ that satisfies Property R.

\begin{thm} \label{thm:wce.S.d.bounds}
Under the assumptions of Proposition~\ref{prop:wce}, if $(X_N)$ is a sequence of $N$-point configurations on $\mathbb{S}^d$ that satisfies Property R, then there exists a constant $c > 0$ independent of $N$ such that for integers $t$ with $t \asymp N^{1/d}$ 
\begin{equation*}
\left[ \wce( \numint[X_{N}]; \mathbb{H}^s( \mathbb{S}^d ) ) \right]^2 \leq \frac{c}{N^2} \sum_{j=1}^{N} \sum_{i=1}^{N} \, \mathcal{K}^{(s)}_t({\PT x}_j \cdot {\PT x}_i) + \mathcal{O}(N^{-2s/d}) \qquad \text{as $N\to\infty$.}
\end{equation*}
\end{thm}

\begin{proof}[Proof of Theorem~\ref{thm:characterization}]
Let $(X_N)$ be a sequence of $N$-point configurations on $\mathbb{S}^d$ satisfying Property R.
If $(X_N)$ is a sequence of QMC designs for $\mathbb{H}^s( \mathbb{S}^d)$, then by Proposition~\ref{prop:necessary.conditions}, for all integers~$t$ 
\begin{equation*}
\frac{1}{N^2}\sum_{j=1}^{N} \sum_{i=1}^{N} \, \mathcal{K}^{(s)}_t({\PT x}_j^t \cdot {\PT x}_i^t) = \mathcal{O}(N^{-2s/d}) \qquad \text{as $N \to \infty$.}
\end{equation*}
In the converse direction, if $(X_N)$ satisfies the above condition, then by Theorem~\ref{thm:wce.S.d.bounds}
\begin{equation*}
\left[ \wce( \numint_{N}; \mathbb{H}^s( \mathbb{S}^2 ) \right]^2 = \mathcal{O}(N^{-2s/d}) \qquad \text{as $N \to \infty$},
\end{equation*}
and hence $(X_N)$ is a sequence of QMC designs for $\mathbb{H}^s( \mathbb{S}^d)$.
\end{proof}

\subsection{Proofs of Section~\ref{sec:uniform.distribution} results}

\begin{proof}[Proof of Proposition~\ref{prop:Kosma.Hlawka.bound}]
We recall that two different kernels for the same $\mathbb{H}^s( \mathbb{S}^d)$ yield worst case errors that can differ by at most constant factors, and also that for the particular case of the canonical kernel $K_{\mathrm{can}}^{(s)}$, the worst case error decreases monotonically with $s$ for
$s>d/2$, cf. relation~\eqref{eq:wce.inequality.A}. Using these facts and \eqref{eq:discrepancy.identies}, which gives the worst-case error with respect to $K_{\mathrm{dist}}$, we have for $s \geq (d + 1)/2$ the following estimates:
\begin{equation*}
\wce( \numint[X_N]; \mathbb{H}^{s}( \mathbb{S}^d ) ) \leq \beta_1 \, \wce( \numint[X_N]; \mathbb{H}^{(d+1)/2}( \mathbb{S}^d ) ) \leq \beta_2 \, D_{\IL_{2}}^{C}( X_N ) \leq \beta_3 \, D_{\IL_{\infty}}^{C}( X_N ), %, \quad f \in \mathbb{H}^{s}( \mathbb{S}^d ),
\end{equation*}
where the positive constants $\beta_1$, $\beta_2$ and $\beta_2$ depend on the chosen norms.
% 
% is introduced because the distance kernel $K_{\mathrm{dist}}$\MARKED{Purple}{, giving rise to the worst-case error appearing on the left-hand side of the identity \eqref{eq:discrepancy.identies},} may be different from $K_{\mathrm{can}}^{(s)}$. 
\end{proof}

\subsection{Proofs of Section~\ref{sec:better.than.average} results}

\begin{proof}[Proof of Theorem~\ref{thm:average}]
Let $\mathcal{A}^L(z)$ be the truncated series
\begin{equation*}
\mathcal{A}^L(z) \DEF \sum_{\ell=1}^L a_\ell \, Z(d,\ell) \, P_\ell^{(d)}(z), \qquad z \in [-1,1].
\end{equation*}
On separating the diagonal and off-diagonal terms of the double sum
\begin{equation*}
\mathcal{A}^L[ X_N ] \DEF \frac{1}{N^2} \sum_{j=1}^N \sum_{i=1}^N \mathcal{A}^L( \PT{x}_j \cdot \PT{x}_i ), \qquad X_N = \{ \PT{x}_1, \dots, \PT{x}_N \} \subset \mathbb{S}^d,
\end{equation*}
we obtain
\begin{align*}
\mathbb{E} \mathcal{A}^L[ X_N ]
&= \int_{\mathbb{S}^d} \cdots \int_{\mathbb{S}^d} \Bigg[ \frac{\mathcal{A}^L(1)}{N} + \frac{1}{N^2} \mathop{\sum_{j=1}^N \sum_{i=1}^N}_{j \neq i} \mathcal{A}^L( \PT{x}_j \cdot \PT{x}_i ) \Bigg] \dd \sigma_d( \PT{x}_1 ) \cdots \dd \sigma_d( \PT{x}_N ) \\
&= \frac{\mathcal{A}^L(1)}{N} + \frac{1}{N^2} \mathop{\sum_{j=1}^N \sum_{i=1}^N}_{j \neq i} \int_{\mathbb{S}^d} \int_{\mathbb{S}^d} \mathcal{A}^L( \PT{x}_j \cdot \PT{x}_i ) \dd \sigma_d( \PT{x}_j ) \dd \sigma_d( \PT{x}_i ) \\
&= \frac{\mathcal{A}^L(1)}{N} + \frac{N ( N - 1 )}{N^2} \sum_{\ell=1}^L a_\ell \sum_{k=1}^{Z(d,\ell)} \left( \int_{\mathbb{S}^d} Y_{\ell,k}( \PT{x} ) \dd \sigma_d( \PT{x} ) \right)^2 \\
&= \frac{\mathcal{A}^L(1)}{N}.
\end{align*}

For constant coefficients $a_1 = a_2 = \cdots = a_L = 1$ this is the result of Theorem~6 in \cite{SlWo2009}.

The quantities $\mathcal{A}^1[X_N], \mathcal{A}^2[X_N], \dots$ form a point-wise non-decreasing sequence of non-negative $(\sigma_d \cdots \sigma_d)$-measurable functions with limit function $\mathcal{A}[X_N]$. By the monotone convergence theorem it follows that
\begin{align*}
\mathbb{E} \mathcal{A}[ X_N ]
&= \int_{\mathbb{S}^d} \cdots \int_{\mathbb{S}^d} \mathcal{A}[ X_N ] \dd \sigma_d( \PT{x}_1 ) \cdots \dd \sigma_d( \PT{x}_N ) \\
&= \lim_{L \to \infty} \mathbb{E} \mathcal{A}^L[ X_N ] = \lim_{L \to \infty} \frac{\mathcal{A}^L(1)}{N} = \frac{\mathcal{A}(1)}{N}.
\end{align*}
This completes the proof.
\end{proof}

\begin{proof}[Proof of Theorem~\ref{thm:equal.area}]
We follow the proof idea leading to \cite[Theorem~2.2]{RaSaZh1994}. 
Let $\mathcal{D}_N = \{D_{j,N}, \dots, D_{N,N} \}$ be an equal area partition of $\mathbb{S}^d$ into subsets with small diameter; that is: $\cup_{j=1}^N D_{j,N} = \mathbb{S}^d$, where $\sigma_d( D_{j,N} \cap D_{k,N} ) = 0$ for all $j, k = 1, \dots, N$ with $j \neq k$ and $\sigma_d( D_{j,N} ) = 1/N$; furthermore, $\diam D_{j,N} \leq c / N^{1/d}$ for some $c$ not depending on $N$. Each $D_{j,N}$ is equipped with the probability measure
\begin{equation} \label{eq:mu.j.N}
\mu_{j,N} \DEF \frac{\sigma_d \big|_{D_{j,N}}}{\sigma_d( D_{j,N} )}.
\end{equation}

Let $d/2 < s < d/2 + 1$. Then the expected value of the squared worst-case error $[\wce( \numint[X_N]; \mathbb{H}^s( \mathbb{S}^d ) )]^2$ for the space $\mathbb{H}^s( \mathbb{S}^d )$  provided with the kernel $K_{\mathrm{gd}}^{(s)}$ given in \eqref{eq:K.gd} when the $j$-th node is chosen randomly from $D_{j,N}$ (with respect to uniform measure on $D_{j,N}$) is given by (see \eqref{eq:dist.kernel.wce.formula})
\begin{align*}
&\mathbb{E} \left[ \left\{ \wce( \numint[X_N]; \mathbb{H}^s( \mathbb{S}^d ) ) \right\}^2 \right] \\
&\phantom{equals}= \int_{D_{1,N}} \cdots \int_{D_{N,N}} \left[ \wce( \numint[X_N]; \mathbb{H}^s( \mathbb{S}^d ) ) \right]^2 \dd \mu_{1,N}( \PT{x}_1 ) \cdots \dd \mu_{N,N}( \PT{x}_N ) \\
&\phantom{equals}= V_{d-2s}( \mathbb{S}^d ) - \frac{1}{N^2} \mathop{\sum_{j=1}^N \sum_{i=1}^N}_{j \neq i} \int_{D_{j,N}} \int_{D_{i,N}} \left| \PT{x}_j - \PT{x}_i \right|^{2s-d} \dd \mu_{j,N}( \PT{x}_j ) \dd \mu_{i,N}( \PT{x}_i ) \\
&\phantom{equals}= V_{d-2s}( \mathbb{S}^d ) - \Bigg[ \int_{\mathbb{S}^d} \int_{\mathbb{S}^d} \left| \PT{x} - \PT{y} \right|^{2s-d} \dd \sigma_d( \PT{x} ) \dd \sigma_d( \PT{y} ) \\
&\phantom{equals= V_{d-2s}( \mathbb{S}^d ) - }- \frac{1}{N^2} \sum_{j=1}^N \int_{D_{j,N}} \int_{D_{j,N}} \left| \PT{x} - \PT{y} \right|^{2s-d} \dd \mu_{j,N}( \PT{x} ) \dd \mu_{j,N}( \PT{y} ) \Bigg].
\end{align*}
Since the first double integral in brackets equals $V_{d-2s}( \mathbb{S}^d )$ (see \eqref{eq:gd.kernel.coeffs}), we deduce that
\begin{align*}
&\mathbb{E} \left[ \left\{ \wce( \numint[X_N]; \mathbb{H}^s( \mathbb{S}^d ) ) \right\}^2 \right] 
= \frac{1}{N^2} \sum_{j=1}^N \int_{D_{j,N}} \int_{D_{j,N}} \left| \PT{x} - \PT{y} \right|^{2s-d} \dd \mu_{j,N}( \PT{x} ) \dd \mu_{j,N}( \PT{y} ) \\
&\phantom{equals}\leq \frac{1}{N^2} \sum_{j=1}^N  \left[ \diam D_{j,N} \right]^{2s-d} \leq \frac{1}{N^2} \sum_{j=1}^N \left[ c N^{-1/d} \right]^{2s-d} = c^{2s-d} N^{-2s/d}.
\end{align*}

The lower bound in \eqref{eq:E.wce.equal.area} follows from Theorem~\ref{thm:wce.lower.bound}. 
\end{proof}

\begin{proof}[Proof of Theorem~\ref{thm:equal.area.lower.bound}]
Let $d/2 + L < s < d/2 + L + 1$ for an integer $L \geq 1$. 
Arguing as in the proof of Theorem~\ref{thm:equal.area}, but using the kernel in \eqref{eq:K.gd.general}, we obtain
\begin{align*}
\mathbb{E} \left[ \left\{ \wce( \numint[X_N]; \mathbb{H}^s( \mathbb{S}^d ) ) \right\}^2 \right]
&= \frac{1}{N^2} \sum_{j=1}^N \int_{D_{j,N}} \int_{D_{j,N}} \left[ \mathcal{Q}_L( 1 ) - \mathcal{Q}_L( \PT{x} \cdot \PT{y} ) \right] \dd \mu_{j,N}( \PT{x} ) \dd \mu_{j,N}( \PT{y} ) \\
&\phantom{=}- \frac{1}{N^2} \sum_{j=1}^N \int_{D_{j,N}} \int_{D_{j,N}} (-1)^{L+1} \left| \PT{x} - \PT{y} \right|^{2s-d} \dd \mu_{j,N}( \PT{x} ) \dd \mu_{j,N}( \PT{y} ),
\end{align*}
where we used the fact that the Laplace-Fourier expansion \eqref{eq:cal.Q.L} only contains Legendre polynomials $P_\ell^{(d)}$ with $\ell \geq 1$ and thus
\begin{equation*}
\frac{1}{N^2} \sum_{j=1}^N \sum_{i=1}^N \int_{D_{j,N}} \int_{D_{i,N}} \mathcal{Q}_L( \PT{x} \cdot \PT{y} ) \dd \mu_{i,N}( \PT{x} ) \dd \mu_{k,N}( \PT{y} ) = \int_{\mathbb{S}^d} \int_{\mathbb{S}^d} \mathcal{Q}_L( \PT{x} \cdot \PT{y} ) \dd \sigma_d( \PT{x} ) \dd \sigma_d( \PT{y} ) = 0.
\end{equation*}

For $L = 1$ the definition of $\mathcal{Q}_L$ given in \eqref{eq:cal.Q.L} and the fact that $P_1^{(d)}(x) = x$ yields
\begin{equation*}
\mathcal{Q}_1( 1 ) - \mathcal{Q}_1( \PT{x} \cdot \PT{y} ) = - \alpha_1^{(s)} Z(d, 1) \left( 2 - 2 \PT{x} \cdot \PT{y} \right) = - \alpha_1^{(s)} \left( d + 1 \right) \left| \PT{x} - \PT{y} \right|^2,
\end{equation*}
where $\alpha_1^{(s)} < 0$ by \eqref{eq:gd.kernel.coeffs}. More generally, using the following hypergeometric function relation for the polynomials $P_\ell^{(d)}(x)$ (see, e.g, \cite[Eq.~18.5.9]{DLMF2010.05.07}),
\begin{equation*}
P_\ell^{(d)}(x) = \Hypergeom{2}{1}{-\ell,\ell+d-1}{d/2}{\frac{1-x}{2}},
\end{equation*}
we have the following representation for $\mathcal{Q}_L( \PT{x} \cdot \PT{y} )$ in terms of even powers of distances:
\begin{align*}
&\mathcal{Q}_L( 1 ) - \mathcal{Q}_L( \PT{x} \cdot \PT{y} ) \\
&\phantom{equals}= \mathcal{Q}_L( 1 ) - \sum_{\ell=1}^L \sum_{p=0}^\ell \left( (-1)^{L+1-\ell} - 1 \right) \alpha_\ell^{(s)} Z(d, \ell) \frac{\Pochhsymb{-\ell}{p} \Pochhsymb{\ell + d - 1}{p}}{\Pochhsymb{d/2}{p} p!} \left( \frac{1 - \PT{x} \cdot \PT{y}}{2} \right)^p \\
&\phantom{equals}= - \sum_{p=1}^\ell \left\{ \sum_{\ell=p}^L \left( (-1)^{L+1-\ell} - 1 \right) \alpha_\ell^{(s)} Z(d, \ell) \Pochhsymb{-\ell}{p} \Pochhsymb{\ell + d - 1}{p} \right\} \frac{1}{\Pochhsymb{d/2}{p} p!} \left( \frac{\left| \PT{x} - \PT{y} \right|}{2} \right)^{2p}.
\end{align*}
For small distances the dominant term is the square of the distance. Since $(-1)^{L+1-\ell} \alpha_\ell^{(s)}$ is positive by \eqref{eq:gd.kernel.coeffs}, it follows that the coefficient of $( | \PT{x} - \PT{y} | / 2 )^2$,
\begin{equation*}
\beta_1^{(s)} \DEF \frac{2}{d} \sum_{\ell=1}^L \left( (-1)^{L+1-\ell} - 1 \right) \alpha_\ell^{(s)} Z(d, \ell) \left[ - \Pochhsymb{-\ell}{1} \right] \Pochhsymb{\ell + d - 1}{1},
\end{equation*}
is positive and therefore
\begin{equation*}
\mathcal{Q}_L( 1 ) - \mathcal{Q}_L( \PT{x} \cdot \PT{y} ) = \beta_1^{(s)} \, \left( \frac{\left| \PT{x} - \PT{y} \right|}{2} \right)^2 + \mathcal{O}( \left| \PT{x} - \PT{y} \right|^4 ) \qquad \text{as $\left| \PT{x} - \PT{y} \right| \to 0$.}
\end{equation*}
Hence,
\begin{equation} \label{eq:expected.value.lower.bound}
\begin{split} 
&\mathbb{E} \left[ \left\{ \wce( \numint[X_N]; \mathbb{H}^s( \mathbb{S}^d ) ) \right\}^2 \right] \\
&\phantom{equals}\geq \frac{\beta_1^{(s)}}{4} \frac{1}{N^2} \sum_{j=1}^N \int_{D_{j,N}} \int_{D_{j,N}} \left| \PT{x} - \PT{y} \right|^2 \dd \mu_{j,N}( \PT{x} ) \dd \mu_{j,N}( \PT{y} ) - \mathcal{R}_N,
\end{split}
\end{equation}
where (as $N \to \infty$)
\begin{equation} \label{eq:expected.value.lower.bound.remainder}
\begin{split} 
\mathcal{R}_N 
&= \mathcal{O}( \frac{1}{N^2} \sum_{j=1}^N \left[ \diam D_{j,N} \right]^4 ) + \mathcal{O}( \frac{1}{N^2} \sum_{j=1}^N \left[ \diam D_{j,N} \right]^{2s-d} ) \\
&= \mathcal{O}( N^{-4/d-1} ) + \mathcal{O}( N^{-2s/d} ).% \qquad \text{as $N \to \infty$}.
\end{split}
\end{equation}

Next, we observe that for any $c^{\prime\prime} > 0$ the following inequalities hold:
\begin{equation} \label{eq:int.inequality}
\begin{split}
&\int_{D_{j,N}} \int_{D_{j,N}} \left| \PT{x} - \PT{y} \right|^{2} \dd \mu_{j,N}( \PT{x} ) \dd \mu_{j,N}( \PT{y} ) %\\
% &\phantom{equals}
\geq \mathop{\int_{D_{j,N}} \int_{D_{j,N}}}_{\left| \PT{x} - \PT{y} \right| > c^{\prime\prime} / N^{1/d}} \left| \PT{x} - \PT{y} \right|^{2} \dd \mu_{j,N}( \PT{x} ) \dd \mu_{j,N}( \PT{y} ) \\
&\phantom{\int_{D_{j,N}} \int_{D_{j,N}} \left| \PT{x} - \PT{y} \right|^{2}}\geq \left( c^{\prime\prime} \right)^{2} N^{-2/d} \, \mathop{\int_{D_{j,N}} \int_{D_{j,N}}}_{\left| \PT{x} - \PT{y} \right| > c^{\prime\prime} / N^{1/d}} \dd \mu_{j,N}( \PT{x} ) \dd \mu_{j,N}( \PT{y} ) \\
&\phantom{\int_{D_{j,N}} \int_{D_{j,N}} \left| \PT{x} - \PT{y} \right|^{2}}= \left( c^{\prime\prime} \right)^{2} N^{-2/d} \, \Bigg\{ 1 - \mathop{\int_{D_{j,N}} \int_{D_{j,N}}}_{\left| \PT{x} - \PT{y} \right| \leq c^{\prime\prime} / N^{1/d}} \dd \mu_{j,N}( \PT{x} ) \dd \mu_{j,N}( \PT{y} ) \Bigg\}.
\end{split}
\end{equation}
Since $\mu_{j,N}$ is a probability measure on $D_{j,N}$ and $\sigma_d( D_{j,N} ) = 1 / N$, we can bound the above double integral by
\begin{equation} \label{eq:D.j.N.integral.bound}
\begin{split}
\mathop{\int_{D_{j,N}} \int_{D_{j,N}}}_{\left| \PT{x} - \PT{y} \right| < c^{\prime\prime} / N^{1/d}} \dd \mu_{j,N}( \PT{x} ) \dd \mu_{j,N}( \PT{y} ) 
&\leq \int_{D_{j,N}} \int_{\mathcal{C}( \PT{x}; \theta^\prime)} \dd \mu_{j,N}( \PT{y} ) \dd \mu_{j,N}( \PT{x} ) \\
&= \mu_{j,N}( D_{j,N} ) \, \mu_{j,N}( \mathcal{C}( \PT{x}; \theta^\prime) ) = \frac{\sigma_d( \mathcal{C}( \PT{x}; \theta^\prime ))}{\sigma_d( D_{j,N} )}, 
\end{split}
\end{equation}
where $2 \sin( \theta^{\prime} / 2 ) = c^{\prime\prime} / N^{1/d}$. An application of the Funk-Hecke formula gives (cf., e.g., \cite{KuSa1998})
\begin{equation*}
\sigma_d( \mathcal{C}( \PT{x}; \theta^\prime ) ) = \frac{1}{d} \frac{\omega_{d-1}}{\omega_d} \left[ 2 \sin( \theta^\prime/2 ) \right]^{d} \left\{ 1 + \mathcal{O}( \left[ 2 \sin( \theta^\prime/2 ) \right]^2 ) \right\}. 
\end{equation*}
Hence, for $N$ sufficiently large,
\begin{equation*}
\frac{\sigma_d( \mathcal{C}( \PT{x}; \theta^\prime )}{\sigma_d( D_{j,N} )} = N \, \frac{1}{d} \frac{\omega_{d-1}}{\omega_d} \left( c^{\prime\prime} / N^{1/d} \right)^{d} \left\{ 1 + \mathcal{O}( \left( c^{\prime\prime} / N^{1/d} \right)^2 ) \right\} \leq 2 \frac{1}{d} \frac{\omega_{d-1}}{\omega_d} \left( c^{\prime\prime} \right)^{d}.
\end{equation*}
By fixing $c^{\prime\prime}$ (which now depends only on $d$) to be sufficiently small, we can always achieve that the double integral on the left-hand side of \eqref{eq:D.j.N.integral.bound} is bounded from above by $1/2$. Therefore, for sufficiently large $N$ and $j = 1, \dots, N$, we deduce from \eqref{eq:int.inequality} that
\begin{equation*}
\int_{D_{j,N}} \int_{D_{j,N}} \left| \PT{x} - \PT{y} \right|^{2} \dd \mu_{j,N}( \PT{x} ) \dd \mu_{j,N}( \PT{y} ) \geq \frac{\left( c^{\prime\prime} \right)^{2}}{2} \, N^{-2/d}.
\end{equation*}
Combining this estimate with \eqref{eq:expected.value.lower.bound} and \eqref{eq:expected.value.lower.bound.remainder} we obtain for $s > d / 2 + 1$ and $2s -d$ not an even positive integer
% Consequently, using the last estimate in \eqref{eq:expected.value.lower.bound} and \eqref{eq:expected.value.lower.bound.remainder}
\begin{equation*}
\begin{split}
\mathbb{E} \left[ \left\{ \wce( \numint[X_N]; \mathbb{H}^s( \mathbb{S}^d ) ) \right\}^2 \right] 
&\geq \frac{\beta_1^{(s)}}{4} \, \frac{\left( c^{\prime\prime} \right)^{2}}{2} \, N^{-2/d-1} + \mathcal{O}( N^{-4/d-1} ) + \mathcal{O}( N^{-2s/d} ) \\
&\geq \beta \, N^{-2/d-1} = \beta \, N^{-2(d/2+1)/d},
\end{split}
\end{equation*}
where the positive constant $\beta$ depends on the $\mathbb{H}^s( \mathbb{S}^d )$-norm 
% $d$, $s$ 
and the partition sequence $(\mathcal{D}_N)$, but is independent of $N$.
\end{proof}

\bibliographystyle{abbrv}
\bibliography{bibliography}
\end{document}